\documentclass{amsart}

\usepackage{amsmath,amssymb,amsthm,a4wide,subcaption}
\usepackage{pgfplots}
\usepackage{graphicx,tikz}
\newtheorem{theorem}{Theorem}

\newtheorem*{proposition}{Proposition}

\newtheorem{lemma}{Lemma}

\theoremstyle{definition}

\theoremstyle{remark}

\begin{document}

\title[]{Ulam sequences and Ulam Sets}

\author[]{Noah Kravitz}
\address{Grace Hopper College, Yale University, New Haven, CT 06511, USA}
\email{noah.kravitz@yale.edu}

\author[]{Stefan Steinerberger}
\address{Department of Mathematics, Yale University, New Haven, CT 06511, USA}
\email{stefan.steinerberger@yale.edu}

\begin{abstract} The Ulam sequence is given by $a_1 =1, a_2 = 2$, and then, for $n \geq 3$, the element $a_n$ is defined as the smallest integer that can be written as the sum of
two distinct earlier elements in a unique way. This gives the sequence $1, 2, 3, 4, 6, 8, 11, 13, 16,  \dots$, which has a mysterious quasi-periodic behavior that is not understood. Ulam's definition naturally extends
 to higher dimensions: for a set of initial vectors $\left\{v_1, \dots, v_k\right\} \subset \mathbb{R}^n$,
we define a sequence by repeatedly adding the smallest elements that can be uniquely written as the sum of two distinct
vectors already in the set. The resulting sets have very rich structure that turns out to be universal for many commuting binary operations.  We give examples of different types of behavior, prove several universality results, and describe new unexplained phenomena.

\end{abstract}

\maketitle
%
%

\vspace{-10pt}

\section{Introduction}
\subsection{Background.}
Stanis\l{}aw Ulam introduced the sequence
 $$1, 2, 3, 4, 6, 8, 11, 13, 16, 18, 26, 28, 36, 38, 47, 48, 53, 57, 62, 69, 72, 77, 82, 87, 97 \dots $$
in a 1964 survey \cite{ulam} on unsolved problems.  The sequence is given by $a_1 =1, a_2 = 2$, after which we iteratively choose the next element to be the smallest integer that can be written as the sum of two distinct earlier elements in a unique way. Ulam asks in \cite{ulam2} whether it is possible to determine the asymptotic density of the sequence (which, empirically, seems to be somewhere around 0.079). At
first glance, this sequence seems somewhat arbitrary and contrived.
\begin{center}
\begin{figure}[h!]
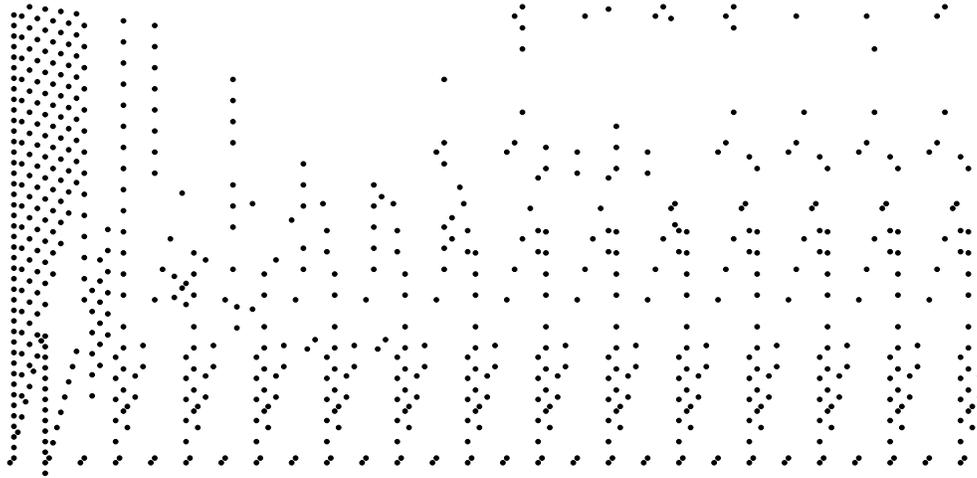


\captionsetup{width = 0.95\textwidth}
\caption{The Ulam set arising from $\left\{(9,0), (0,9), (1,13)\right\}$ shows both chaotic and regular behavior. In \S4 we will prove that it is periodic in both the $x-$ and $y-$directions.}
\label{fig:twodlattice}
\end{figure}
\end{center}
\vspace{-15pt}
 Ulam himself is not very clear about his motivation, and the original text reads only:
\begin{quote}
Another one-dimensional sequence was studied recently. This one was defined purely additively: one starts, say, with integers 1,2 and considers in turn all integers
which are obtainable as the sum of two different ones previously defined but only if they are so expressible in a unique way. The sequence would start
as follows:
$$ 1,2,3,4,6,8,11,13, \dots;$$
even sequences this simple present problems. (Ulam, 1964)
\end{quote}

It is trivial to see that $a_n \leq F_{n+1}$, where $F_n$ is the $n-$th Fibonacci number. This is, in fact, the only rigorously proven statement about the Ulam sequence that we are aware of.
Different initial values $a_1, a_2$ can give rise to more structured sequences \cite{finch1, finch2, q}: for some, the sequence of consecutive differences $a_{n+1} - a_{n}$ is eventually periodic. Ulam's original
sequence beginning with $1,2$ does not seem to become periodic: Knuth \cite{oeis} remarks that $a_{4953}-a_{4952}=262$ and $a_{18858} - a_{18857} = 315$. The understanding in the literature is that the sequence `does not appear to follow any recognizable pattern' \cite{finch3} and is `quite erratic' \cite{schm}. Other initial values sometimes give rise to more regular sequences, a phenomenon investigated by Cassaigne \& Finch \cite{cas}, Finch \cite{finch1, finch3, finch2}, Queneau \cite{q} and, proving a conjecture of Finch, Schmerl \& Spiegel \cite{schm}.
\begin{center}
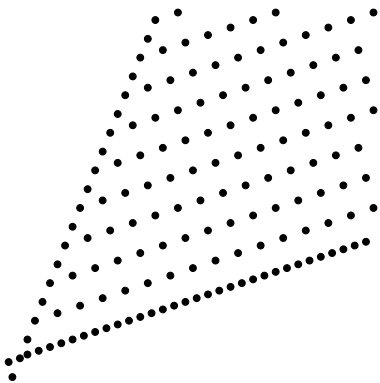
\begin{figure}[h!]
\begin{tikzpicture}[scale=0.05]
\filldraw (3,1) circle (0.9cm);
\filldraw (2,5) circle (0.9cm);
\filldraw (5,6) circle (0.9cm);
\filldraw (7,7) circle (0.9cm);
\filldraw (7,11) circle (0.9cm);
\filldraw (10,8) circle (0.9cm);
\filldraw (9,16) circle (0.9cm);
\filldraw (13,9) circle (0.9cm);
\filldraw (11,21) circle (0.9cm);
\filldraw (16,10) circle (0.9cm);
\filldraw (13,26) circle (0.9cm);
\filldraw (19,11) circle (0.9cm);
\filldraw (15,31) circle (0.9cm);
\filldraw (22,12) circle (0.9cm);
\filldraw (17,36) circle (0.9cm);
\filldraw (25,13) circle (0.9cm);
\filldraw (19,41) circle (0.9cm);
\filldraw (28,14) circle (0.9cm);
\filldraw (21,46) circle (0.9cm);
\filldraw (31,15) circle (0.9cm);
\filldraw (23,51) circle (0.9cm);
\filldraw (34,16) circle (0.9cm);
\filldraw (25,56) circle (0.9cm);
\filldraw (37,17) circle (0.9cm);
\filldraw (27,61) circle (0.9cm);
\filldraw (40,18) circle (0.9cm);
\filldraw (29,66) circle (0.9cm);
\filldraw (43,19) circle (0.9cm);
\filldraw (31,71) circle (0.9cm);
\filldraw (46,20) circle (0.9cm);
\filldraw (33,76) circle (0.9cm);
\filldraw (49,21) circle (0.9cm);
\filldraw (35,81) circle (0.9cm);
\filldraw (52,22) circle (0.9cm);
\filldraw (37,86) circle (0.9cm);
\filldraw (55,23) circle (0.9cm);
\filldraw (39,91) circle (0.9cm);
\filldraw (58,24) circle (0.9cm);
\filldraw (41,96) circle (0.9cm);
\filldraw (61,25) circle (0.9cm);
\filldraw (64,26) circle (0.9cm);
\filldraw (67,27) circle (0.9cm);
\filldraw (70,28) circle (0.9cm);
\filldraw (73,29) circle (0.9cm);
\filldraw (76,30) circle (0.9cm);
\filldraw (79,31) circle (0.9cm);
\filldraw (82,32) circle (0.9cm);
\filldraw (85,33) circle (0.9cm);
\filldraw (88,34) circle (0.9cm);
\filldraw (91,35) circle (0.9cm);
\filldraw (94,36) circle (0.9cm);
\filldraw (97,37) circle (0.9cm);
\filldraw (15,18) circle (0.9cm);
\filldraw (19,28) circle (0.9cm);
\filldraw (23,38) circle (0.9cm);
\filldraw (27,48) circle (0.9cm);
\filldraw (31,58) circle (0.9cm);
\filldraw (35,68) circle (0.9cm);
\filldraw (39,78) circle (0.9cm);
\filldraw (43,88) circle (0.9cm);
\filldraw (47,98) circle (0.9cm);
\filldraw (21,20) circle (0.9cm);
\filldraw (25,30) circle (0.9cm);
\filldraw (29,40) circle (0.9cm);
\filldraw (33,50) circle (0.9cm);
\filldraw (37,60) circle (0.9cm);
\filldraw (41,70) circle (0.9cm);
\filldraw (45,80) circle (0.9cm);
\filldraw (49,90) circle (0.9cm);
\filldraw (27,22) circle (0.9cm);
\filldraw (31,32) circle (0.9cm);
\filldraw (35,42) circle (0.9cm);
\filldraw (39,52) circle (0.9cm);
\filldraw (43,62) circle (0.9cm);
\filldraw (47,72) circle (0.9cm);
\filldraw (51,82) circle (0.9cm);
\filldraw (55,92) circle (0.9cm);
\filldraw (33,24) circle (0.9cm);
\filldraw (37,34) circle (0.9cm);
\filldraw (41,44) circle (0.9cm);
\filldraw (45,54) circle (0.9cm);
\filldraw (49,64) circle (0.9cm);
\filldraw (53,74) circle (0.9cm);
\filldraw (57,84) circle (0.9cm);
\filldraw (61,94) circle (0.9cm);
\filldraw (39,26) circle (0.9cm);
\filldraw (43,36) circle (0.9cm);
\filldraw (47,46) circle (0.9cm);
\filldraw (51,56) circle (0.9cm);
\filldraw (55,66) circle (0.9cm);
\filldraw (59,76) circle (0.9cm);
\filldraw (63,86) circle (0.9cm);
\filldraw (67,96) circle (0.9cm);
\filldraw (45,28) circle (0.9cm);
\filldraw (49,38) circle (0.9cm);
\filldraw (53,48) circle (0.9cm);
\filldraw (57,58) circle (0.9cm);
\filldraw (61,68) circle (0.9cm);
\filldraw (65,78) circle (0.9cm);
\filldraw (69,88) circle (0.9cm);
\filldraw (73,98) circle (0.9cm);
\filldraw (51,30) circle (0.9cm);
\filldraw (55,40) circle (0.9cm);
\filldraw (59,50) circle (0.9cm);
\filldraw (63,60) circle (0.9cm);
\filldraw (67,70) circle (0.9cm);
\filldraw (71,80) circle (0.9cm);
\filldraw (75,90) circle (0.9cm);
\filldraw (57,32) circle (0.9cm);
\filldraw (61,42) circle (0.9cm);
\filldraw (65,52) circle (0.9cm);
\filldraw (69,62) circle (0.9cm);
\filldraw (73,72) circle (0.9cm);
\filldraw (77,82) circle (0.9cm);
\filldraw (81,92) circle (0.9cm);
\filldraw (63,34) circle (0.9cm);
\filldraw (67,44) circle (0.9cm);
\filldraw (71,54) circle (0.9cm);
\filldraw (75,64) circle (0.9cm);
\filldraw (79,74) circle (0.9cm);
\filldraw (83,84) circle (0.9cm);
\filldraw (87,94) circle (0.9cm);
\filldraw (69,36) circle (0.9cm);
\filldraw (73,46) circle (0.9cm);
\filldraw (77,56) circle (0.9cm);
\filldraw (81,66) circle (0.9cm);
\filldraw (85,76) circle (0.9cm);
\filldraw (89,86) circle (0.9cm);
\filldraw (93,96) circle (0.9cm);
\filldraw (75,38) circle (0.9cm);
\filldraw (79,48) circle (0.9cm);
\filldraw (83,58) circle (0.9cm);
\filldraw (87,68) circle (0.9cm);
\filldraw (91,78) circle (0.9cm);
\filldraw (95,88) circle (0.9cm);
\filldraw (99,98) circle (0.9cm);
\filldraw (81,40) circle (0.9cm);
\filldraw (85,50) circle (0.9cm);
\filldraw (89,60) circle (0.9cm);
\filldraw (93,70) circle (0.9cm);
\filldraw (97,80) circle (0.9cm);
\filldraw (87,42) circle (0.9cm);
\filldraw (91,52) circle (0.9cm);
\filldraw (95,62) circle (0.9cm);
\filldraw (99,72) circle (0.9cm);
\filldraw (93,44) circle (0.9cm);
\filldraw (97,54) circle (0.9cm);
\filldraw (99,46) circle (0.9cm);
\end{tikzpicture}
\captionsetup{width = 0.9\textwidth}
\caption{The set arising from $\left\{(2,5), (3,1)\right\}$ creates a regular pattern.  The regularity of this set is a consequence of a more general result in \S 2.1.}
\label{fig:twodlattice}
\end{figure}
\end{center}
\vspace{-20pt}
\subsection{The Hidden Structure.} The second author \cite{stein} recently discovered, more or less by accident, that the Ulam sequence has a nontrivial interaction with Fourier series: more precisely, if
we let $(a_n)_{n=1}^{\infty}$ denote the sequence starting with $a_1 = 1, a_2 = 2$, then, empirically, there seems to exist a real number $\alpha \sim 2.571447\dots$ such that
$$ \sum_{n=1}^{N}{ \cos{(\alpha a_n)}} \sim -0.79 N.$$
If the sequence were truly random, we would expect the sum to be at scale $\sim\sqrt{N}$. Thus, this finding indicates the presence of a strong intrinsic structure in the sequence $\alpha a_n$. The
underlying structure is very rigid: for the first $10^7$ terms of the sequence,
$$ \cos{\left( 2.5714474995~ a_n\right)} < 0 \qquad \mbox{for all}~a_n \notin \left\{2, 3, 47, 69\right\}.$$
This type of structure usually indicates periodic behavior, but the sequence does not seem to have periodic behavior of any kind -- the phenomenon, an unusual
connection between additive structures in $\mathbb{N}$ and quasi-periodic behavior, is not understood and also occurs for many other initial conditions (although the arising constants vary). One interesting byproduct
of this discovery is the development of an algorithm by Philip Gibbs \cite{gibbs} (see also the description of Knuth \cite{knuth}) which is much faster than previously existing algorithms whenever
such phenomena are present. This allowed the verification of the presence of the phenomenon for the first $10^9$ elements of the sequence. Daniel Ross studied various aspects of the phenomenon in his 2016 PhD thesis  \cite{ross}. Hinman, Kuca, Schlesinger \& Sheydvasser \cite{hinman} recently undertook an in-depth study of the Ulam sequence and uncovered several striking new properties as well as partially answering
some of the questions raised in the present paper.  We also refer the reader to the recent work of Kuca \cite{kuca} on a related sequence.

\subsection{The Big Picture.} If we return to the original setup, we see that the Ulam sequence is given by a very simple greedy algorithm: to find the next element, perform all possible binary
operations on the given set and add the smallest element with a unique representation. Obviously, this scheme can be implemented for any type of set equipped with
a binary operation and a notion of size.  This generalization gives rise to an incredibly rich structure, which we discuss in this paper -- 
we focus our investigation on objects in $\mathbb{R}^2$ and $\mathbb{R}^3$ equipped with the standard addition.  (We do show, however, that the dynamics of lattice points is universal,
and we also describe the behavior of many Ulam sets over algebraic objects equipped with a notion of size and an associative and commutative binary operation.) Generalizing the
classical Ulam sequence is very much in the spirit of Ulam's musings; directly after introducing the sequence, he writes in a passage that may not be well known:
\begin{quote}
For two dimensions one can imagine the lattice of all integral valued points or the division of the plane into equilateral triangles (the hexagonal division). Starting with one
or a finite number of points of such a subdivision, one can `grow' new points defined recursively by, for example, including a new point if it forms with two previously
defined points the third vertex of a triangle, but only doing it in the case where it is uniquely so related to a previous pair; in other words, we don't grow a point if it should be a vertex
of two triangles with different pairs previously taken. Apparently the properties of the figure growing by this definition are difficult to ascertain. For example [...] it is not easy to decide whether or not there will be infinitely long side branches coming off the `stems'. (Ulam, 1964)
\end{quote} 

It is evident that  Ulam himself was considering a generalization of a more geometric flavor. Although this line of inquiry is potentially interesting, we chose to pursue a more algebraic generalization, and the rest of this paper is 
devoted to studying the additive analogue for vectors. However, the question of `whether or not there will be infinitely long side branches coming off the stems' also arises
naturally in our setup (see Figure 3 below and \S 4).

\begin{center}
\begin{figure}[h!]
\begin{tikzpicture}[scale=0.12]

\filldraw (1,0) circle (0.35cm);
\filldraw (0,1) circle (0.35cm);
\filldraw (1,1) circle (0.35cm);
\filldraw (2,0) circle (0.35cm);
\filldraw (1,2) circle (0.35cm);
\filldraw (2,2) circle (0.35cm);
\filldraw (3,0) circle (0.35cm);
\filldraw (1,3) circle (0.35cm);
\filldraw (4,0) circle (0.35cm);
\filldraw (1,4) circle (0.35cm);
\filldraw (4,3) circle (0.35cm);
\filldraw (5,1) circle (0.35cm);
\filldraw (1,5) circle (0.35cm);
\filldraw (6,0) circle (0.35cm);
\filldraw (1,6) circle (0.35cm);
\filldraw (4,5) circle (0.35cm);
\filldraw (1,7) circle (0.35cm);
\filldraw (6,4) circle (0.35cm);
\filldraw (7,2) circle (0.35cm);
\filldraw (8,0) circle (0.35cm);
\filldraw (4,7) circle (0.35cm);
\filldraw (1,8) circle (0.35cm);
\filldraw (6,6) circle (0.35cm);
\filldraw (1,9) circle (0.35cm);
\filldraw (9,3) circle (0.35cm);
\filldraw (4,9) circle (0.35cm);
\filldraw (6,8) circle (0.35cm);
\filldraw (1,10) circle (0.35cm);
\filldraw (9,5) circle (0.35cm);
\filldraw (11,0) circle (0.35cm);
\filldraw (1,11) circle (0.35cm);
\filldraw (11,2) circle (0.35cm);
\filldraw (9,7) circle (0.35cm);
\filldraw (6,10) circle (0.35cm);
\filldraw (4,11) circle (0.35cm);
\filldraw (1,12) circle (0.35cm);
\filldraw (12,1) circle (0.35cm);
\filldraw (9,9) circle (0.35cm);
\filldraw (13,0) circle (0.35cm);
\filldraw (1,13) circle (0.35cm);
\filldraw (6,12) circle (0.35cm);
\filldraw (4,13) circle (0.35cm);
\filldraw (1,14) circle (0.35cm);
\filldraw (9,11) circle (0.35cm);
\filldraw (14,5) circle (0.35cm);
\filldraw (1,15) circle (0.35cm);
\filldraw (15,1) circle (0.35cm);
\filldraw (6,14) circle (0.35cm);
\filldraw (4,15) circle (0.35cm);
\filldraw (14,7) circle (0.35cm);
\filldraw (9,13) circle (0.35cm);
\filldraw (16,0) circle (0.35cm);
\filldraw (1,16) circle (0.35cm);
\filldraw (16,2) circle (0.35cm);
\filldraw (15,6) circle (0.35cm);
\filldraw (14,9) circle (0.35cm);
\filldraw (1,17) circle (0.35cm);
\filldraw (6,16) circle (0.35cm);
\filldraw (4,17) circle (0.35cm);
\filldraw (9,15) circle (0.35cm);
\filldraw (14,11) circle (0.35cm);
\filldraw (18,0) circle (0.35cm);
\filldraw (1,18) circle (0.35cm);
\filldraw (6,18) circle (0.35cm);
\filldraw (18,6) circle (0.35cm);
\filldraw (1,19) circle (0.35cm);
\filldraw (14,13) circle (0.35cm);
\filldraw (9,17) circle (0.35cm);
\filldraw (4,19) circle (0.35cm);
\filldraw (19,5) circle (0.35cm);
\filldraw (20,1) circle (0.35cm);
\filldraw (1,20) circle (0.35cm);
\filldraw (14,15) circle (0.35cm);
\filldraw (6,20) circle (0.35cm);
\filldraw (1,21) circle (0.35cm);
\filldraw (9,19) circle (0.35cm);
\filldraw (21,2) circle (0.35cm);
\filldraw (4,21) circle (0.35cm);
\filldraw (22,1) circle (0.35cm);
\filldraw (1,22) circle (0.35cm);
\filldraw (14,17) circle (0.35cm);
\filldraw (20,10) circle (0.35cm);
\filldraw (6,22) circle (0.35cm);
\filldraw (9,21) circle (0.35cm);
\filldraw (1,23) circle (0.35cm);
\filldraw (20,12) circle (0.35cm);
\filldraw (4,23) circle (0.35cm);
\filldraw (14,19) circle (0.35cm);
\filldraw (1,24) circle (0.35cm);
\filldraw (24,3) circle (0.35cm);
\filldraw (20,14) circle (0.35cm);
\filldraw (23,9) circle (0.35cm);
\filldraw (9,23) circle (0.35cm);
\filldraw (6,24) circle (0.35cm);
\filldraw (1,25) circle (0.35cm);
\filldraw (14,21) circle (0.35cm);
\filldraw (4,25) circle (0.35cm);
\filldraw (25,4) circle (0.35cm);
\filldraw (23,11) circle (0.35cm);
\filldraw (20,16) circle (0.35cm);
\filldraw (25,6) circle (0.35cm);
\filldraw (26,0) circle (0.35cm);
\filldraw (1,26) circle (0.35cm);
\filldraw (23,13) circle (0.35cm);
\filldraw (9,25) circle (0.35cm);
\filldraw (6,26) circle (0.35cm);
\filldraw (20,18) circle (0.35cm);
\filldraw (25,10) circle (0.35cm);
\filldraw (14,23) circle (0.35cm);
\filldraw (1,27) circle (0.35cm);
\filldraw (27,1) circle (0.35cm);
\filldraw (4,27) circle (0.35cm);
\filldraw (23,15) circle (0.35cm);
\filldraw (25,12) circle (0.35cm);
\filldraw (28,0) circle (0.35cm);
\filldraw (1,28) circle (0.35cm);
\filldraw (20,20) circle (0.35cm);
\filldraw (9,27) circle (0.35cm);
\filldraw (23,17) circle (0.35cm);
\filldraw (6,28) circle (0.35cm);
\filldraw (25,14) circle (0.35cm);
\filldraw (14,25) circle (0.35cm);
\filldraw (1,29) circle (0.35cm);
\filldraw (4,29) circle (0.35cm);
\filldraw (29,5) circle (0.35cm);
\filldraw (25,16) circle (0.35cm);
\filldraw (20,22) circle (0.35cm);
\filldraw (23,19) circle (0.35cm);
\filldraw (1,30) circle (0.35cm);
\filldraw (30,2) circle (0.35cm);
\filldraw (9,29) circle (0.35cm);
\filldraw (14,27) circle (0.35cm);
\filldraw (6,30) circle (0.35cm);
\filldraw (30,6) circle (0.35cm);
\filldraw (25,18) circle (0.35cm);
\filldraw (1,31) circle (0.35cm);
\filldraw (30,8) circle (0.35cm);
\filldraw (31,2) circle (0.35cm);
\filldraw (23,21) circle (0.35cm);
\filldraw (20,24) circle (0.35cm);
\filldraw (4,31) circle (0.35cm);
\filldraw (30,10) circle (0.35cm);
\filldraw (25,20) circle (0.35cm);
\filldraw (1,32) circle (0.35cm);
\filldraw (14,29) circle (0.35cm);
\filldraw (9,31) circle (0.35cm);
\filldraw (30,12) circle (0.35cm);
\filldraw (23,23) circle (0.35cm);
\filldraw (6,32) circle (0.35cm);
\filldraw (20,26) circle (0.35cm);
\filldraw (1,33) circle (0.35cm);
\filldraw (30,14) circle (0.35cm);
\filldraw (4,33) circle (0.35cm);
\filldraw (25,22) circle (0.35cm);
\filldraw (23,25) circle (0.35cm);
\filldraw (30,16) circle (0.35cm);
\filldraw (14,31) circle (0.35cm);
\filldraw (1,34) circle (0.35cm);
\filldraw (34,3) circle (0.35cm);
\filldraw (9,33) circle (0.35cm);
\filldraw (33,9) circle (0.35cm);
\filldraw (34,5) circle (0.35cm);
\filldraw (20,28) circle (0.35cm);
\filldraw (6,34) circle (0.35cm);
\filldraw (25,24) circle (0.35cm);
\filldraw (33,11) circle (0.35cm);
\filldraw (30,18) circle (0.35cm);
\filldraw (1,35) circle (0.35cm);
\filldraw (4,35) circle (0.35cm);
\filldraw (35,4) circle (0.35cm);
\filldraw (23,27) circle (0.35cm);
\filldraw (33,13) circle (0.35cm);
\filldraw (14,33) circle (0.35cm);
\filldraw (36,0) circle (0.35cm);
\filldraw (1,36) circle (0.35cm);
\filldraw (20,30) circle (0.35cm);
\filldraw (30,20) circle (0.35cm);
\filldraw (36,2) circle (0.35cm);
\filldraw (25,26) circle (0.35cm);
\filldraw (9,35) circle (0.35cm);
\filldraw (33,15) circle (0.35cm);
\filldraw (6,36) circle (0.35cm);
\filldraw (23,29) circle (0.35cm);
\filldraw (37,1) circle (0.35cm);
\filldraw (1,37) circle (0.35cm);
\filldraw (33,17) circle (0.35cm);
\filldraw (30,22) circle (0.35cm);
\filldraw (4,37) circle (0.35cm);
\filldraw (25,28) circle (0.35cm);
\filldraw (14,35) circle (0.35cm);
\filldraw (20,32) circle (0.35cm);
\filldraw (38,0) circle (0.35cm);
\filldraw (1,38) circle (0.35cm);
\filldraw (9,37) circle (0.35cm);
\filldraw (33,19) circle (0.35cm);
\filldraw (30,24) circle (0.35cm);
\filldraw (6,38) circle (0.35cm);
\filldraw (23,31) circle (0.35cm);
\filldraw (1,39) circle (0.35cm);
\filldraw (25,30) circle (0.35cm);
\filldraw (38,9) circle (0.35cm);
\filldraw (33,21) circle (0.35cm);
\filldraw (4,39) circle (0.35cm);
\filldraw (20,34) circle (0.35cm);
\filldraw (14,37) circle (0.35cm);
\filldraw (38,11) circle (0.35cm);
\filldraw (30,26) circle (0.35cm);
\filldraw (1,40) circle (0.35cm);
\filldraw (9,39) circle (0.35cm);
\filldraw (38,13) circle (0.35cm);
\filldraw (33,23) circle (0.35cm);
\filldraw (23,33) circle (0.35cm);
\filldraw (6,40) circle (0.35cm);
\filldraw (25,32) circle (0.35cm);
\filldraw (1,41) circle (0.35cm);
\filldraw (30,28) circle (0.35cm);
\filldraw (41,2) circle (0.35cm);
\filldraw (20,36) circle (0.35cm);
\filldraw (4,41) circle (0.35cm);
\filldraw (40,10) circle (0.35cm);
\filldraw (41,5) circle (0.35cm);
\filldraw (33,25) circle (0.35cm);
\filldraw (14,39) circle (0.35cm);
\filldraw (41,7) circle (0.35cm);
\filldraw (40,12) circle (0.35cm);
\filldraw (23,35) circle (0.35cm);
\filldraw (9,41) circle (0.35cm);
\filldraw (1,42) circle (0.35cm);
\filldraw (25,34) circle (0.35cm);
\filldraw (30,30) circle (0.35cm);
\filldraw (6,42) circle (0.35cm);
\filldraw (33,27) circle (0.35cm);
\filldraw (20,38) circle (0.35cm);
\filldraw (1,43) circle (0.35cm);
\filldraw (40,16) circle (0.35cm);
\filldraw (4,43) circle (0.35cm);
\filldraw (14,41) circle (0.35cm);
\filldraw (23,37) circle (0.35cm);
\filldraw (25,36) circle (0.35cm);
\filldraw (40,18) circle (0.35cm);
\filldraw (30,32) circle (0.35cm);
\filldraw (9,43) circle (0.35cm);
\filldraw (33,29) circle (0.35cm);
\filldraw (1,44) circle (0.35cm);
\filldraw (44,3) circle (0.35cm);
\filldraw (44,5) circle (0.35cm);
\filldraw (43,11) circle (0.35cm);
\filldraw (6,44) circle (0.35cm);
\filldraw (44,8) circle (0.35cm);
\filldraw (20,40) circle (0.35cm);
\filldraw (40,20) circle (0.35cm);
\filldraw (43,13) circle (0.35cm);
\filldraw (1,45) circle (0.35cm);
\filldraw (4,45) circle (0.35cm);
\filldraw (45,4) circle (0.35cm);
\filldraw (14,43) circle (0.35cm);
\filldraw (23,39) circle (0.35cm);
\filldraw (33,31) circle (0.35cm);
\filldraw (30,34) circle (0.35cm);
\filldraw (25,38) circle (0.35cm);
\filldraw (43,15) circle (0.35cm);
\filldraw (9,45) circle (0.35cm);
\filldraw (46,1) circle (0.35cm);
\filldraw (1,46) circle (0.35cm);
\filldraw (43,17) circle (0.35cm);
\filldraw (6,46) circle (0.35cm);
\filldraw (20,42) circle (0.35cm);
\filldraw (33,33) circle (0.35cm);
\filldraw (30,36) circle (0.35cm);
\filldraw (47,0) circle (0.35cm);
\filldraw (1,47) circle (0.35cm);
\filldraw (23,41) circle (0.35cm);
\filldraw (14,45) circle (0.35cm);
\filldraw (4,47) circle (0.35cm);
\filldraw (25,40) circle (0.35cm);
\filldraw (9,47) circle (0.35cm);
\filldraw (48,0) circle (0.35cm);
\filldraw (1,48) circle (0.35cm);
\filldraw (33,35) circle (0.35cm);
\filldraw (20,44) circle (0.35cm);
\filldraw (6,48) circle (0.35cm);
\filldraw (30,38) circle (0.35cm);
\filldraw (23,43) circle (0.35cm);
\filldraw (25,42) circle (0.35cm);
\filldraw (1,49) circle (0.35cm);
\filldraw (14,47) circle (0.35cm);
\filldraw (49,3) circle (0.35cm);
\filldraw (4,49) circle (0.35cm);
\filldraw (33,37) circle (0.35cm);
\filldraw (9,49) circle (0.35cm);
\filldraw (30,40) circle (0.35cm);
\filldraw (1,50) circle (0.35cm);
\filldraw (49,10) circle (0.35cm);
\filldraw (50,2) circle (0.35cm);
\filldraw (20,46) circle (0.35cm);
\filldraw (6,50) circle (0.35cm);
\filldraw (49,12) circle (0.35cm);
\filldraw (23,45) circle (0.35cm);
\filldraw (25,44) circle (0.35cm);
\filldraw (49,14) circle (0.35cm);
\filldraw (14,49) circle (0.35cm);
\filldraw (51,1) circle (0.35cm);
\filldraw (33,39) circle (0.35cm);
\filldraw (49,16) circle (0.35cm);
\filldraw (30,42) circle (0.35cm);
\filldraw (20,48) circle (0.35cm);
\filldraw (49,18) circle (0.35cm);
\filldraw (23,47) circle (0.35cm);
\filldraw (25,46) circle (0.35cm);
\filldraw (33,41) circle (0.35cm);
\filldraw (49,20) circle (0.35cm);
\filldraw (53,0) circle (0.35cm);
\filldraw (53,2) circle (0.35cm);
\filldraw (30,44) circle (0.35cm);
\filldraw (49,22) circle (0.35cm);
\filldraw (20,50) circle (0.35cm);
\filldraw (25,48) circle (0.35cm);
\filldraw (23,49) circle (0.35cm);
\filldraw (33,43) circle (0.35cm);
\filldraw (49,24) circle (0.35cm);
\filldraw (30,46) circle (0.35cm);
\filldraw (49,26) circle (0.35cm);
\filldraw (33,45) circle (0.35cm);
\filldraw (25,50) circle (0.35cm);
\filldraw (56,1) circle (0.35cm);
\filldraw (49,28) circle (0.35cm);
\filldraw (56,7) circle (0.35cm);
\filldraw (30,48) circle (0.35cm);
\filldraw (57,0) circle (0.35cm);
\filldraw (57,6) circle (0.35cm);
\filldraw (33,47) circle (0.35cm);
\filldraw (49,30) circle (0.35cm);
\filldraw (58,5) circle (0.35cm);
\filldraw (30,50) circle (0.35cm);
\filldraw (49,32) circle (0.35cm);
\filldraw (33,49) circle (0.35cm);
\filldraw (49,34) circle (0.35cm);
\filldraw (56,21) circle (0.35cm);
\filldraw (60,2) circle (0.35cm);
\filldraw (56,23) circle (0.35cm);
\filldraw (49,36) circle (0.35cm);
\filldraw (60,11) circle (0.35cm);
\filldraw (56,25) circle (0.35cm);
\filldraw (60,13) circle (0.35cm);
\filldraw (62,0) circle (0.35cm);
\filldraw (49,38) circle (0.35cm);
\filldraw (56,27) circle (0.35cm);
\filldraw (62,6) circle (0.35cm);
\filldraw (60,17) circle (0.35cm);
\filldraw (60,19) circle (0.35cm);
\filldraw (56,29) circle (0.35cm);
\filldraw (63,5) circle (0.35cm);
\filldraw (49,40) circle (0.35cm);
\filldraw (56,31) circle (0.35cm);
\filldraw (60,23) circle (0.35cm);
\filldraw (49,42) circle (0.35cm);
\filldraw (64,10) circle (0.35cm);
\filldraw (56,33) circle (0.35cm);
\filldraw (60,25) circle (0.35cm);
\filldraw (64,12) circle (0.35cm);
\filldraw (65,4) circle (0.35cm);
\filldraw (60,27) circle (0.35cm);
\filldraw (49,44) circle (0.35cm);
\filldraw (65,11) circle (0.35cm);
\filldraw (56,35) circle (0.35cm);
\filldraw (60,29) circle (0.35cm);
\filldraw (56,37) circle (0.35cm);
\filldraw (49,46) circle (0.35cm);
\filldraw (67,6) circle (0.35cm);
\filldraw (60,31) circle (0.35cm);
\filldraw (56,39) circle (0.35cm);
\filldraw (68,7) circle (0.35cm);
\filldraw (60,33) circle (0.35cm);
\filldraw (49,48) circle (0.35cm);
\filldraw (68,9) circle (0.35cm);
\filldraw (69,0) circle (0.35cm);
\filldraw (69,6) circle (0.35cm);
\filldraw (56,41) circle (0.35cm);
\filldraw (60,35) circle (0.35cm);
\filldraw (49,50) circle (0.35cm);
\filldraw (70,2) circle (0.35cm);
\filldraw (60,37) circle (0.35cm);
\filldraw (56,43) circle (0.35cm);
\filldraw (60,39) circle (0.35cm);
\filldraw (56,45) circle (0.35cm);
\filldraw (60,41) circle (0.35cm);
\filldraw (56,47) circle (0.35cm);
\filldraw (60,43) circle (0.35cm);
\filldraw (56,49) circle (0.35cm);
\filldraw (60,45) circle (0.35cm);
\filldraw (60,47) circle (0.35cm);
\filldraw (60,49) circle (0.35cm);
\end{tikzpicture}
\captionsetup{width = 0.95\textwidth}
\caption{The set arising from $\left\{(1,0), (2,0), (0,1)\right\}$, the classical Ulam sequence on the $x-$axis augmented by $(0,1)$. Regular columns branch
away from the irregular sequence.}
\label{fig:ulamsmall}
\end{figure}
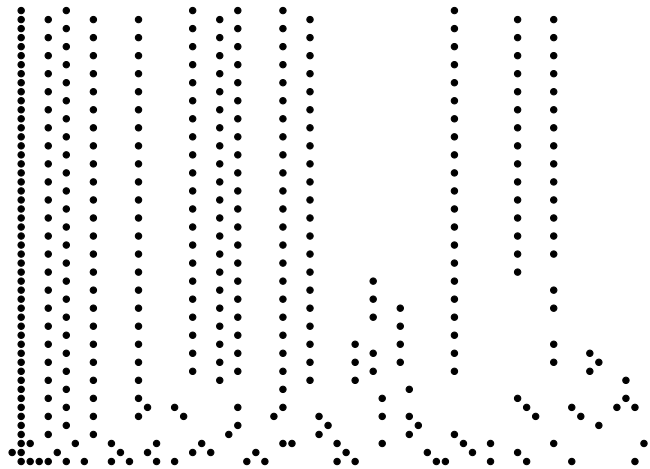
\end{center}
\vspace{-10pt}
Figure \ref{fig:ulamsmall}, which shows the Ulam set generated by $\left\{(1,0), (2,0), (0,1)\right\}$ (see Figure \ref{fig:ulammed} for a larger scale), motivated much of our inquiry into the internal structure of these graphs. In particular, this example recreates the classical Ulam
sequence on the $x-$axis (about which little is known). One main result of this paper (\S 4) implies that for each fixed value of $x\in \mathbb{N}$, the $y-$coordinates of elements in the set are either bounded  in size or eventually become periodic (where the period is some power of 2).

\subsection{Ulam Sets}
We define Ulam sequences in $\mathbb{R}_{\geq 0}^n$ (the set of nonzero vectors all of whose components are nonnegative) by specifying a set of initial vectors $\left\{v_1, \dots, v_k\right\}$ and then repeatedly
adding the smallest vector (with respect to Euclidean norm; see \S 3) that can be uniquely written as the sum of two distinct vectors already in the set.  At any stage of the construction, there may be two or more vectors with the same size each with a unique representation.  In this case, we can simply add all of them at once; having initial conditions contained in $\mathbb{R}^n_{\geq 0}$ guarantees that they can be added one by one in any arbitrary order without affecting the representations of the others. However, this ambiguity makes it clear that there can be no canonical notion of sequence, which is
why we will refer to these objects as unordered \textit{Ulam sets}.\\

\textit{Remarks}.
\begin{enumerate}
\item It is not clear how one would define an Ulam sequence allowing initial elements with negative initial components.  Even for a one-dimensional Ulam-type sequence, allowing negative initial elements proves difficult.  The greedy algorithm does
not work because two numbers with large absolute value can sum to a number with smaller absolute value.  This possibility can potentially destroy the set property of uniqueness of representation retroactively, and the order in which we add vectors tied for smallest becomes a nontrivial consideration. We thus restrict ourselves in the remainder of this paper to nonzero initial elements $\left\{v_1, \dots, v_k\right\} \subset \mathbb{R}_{\geq 0}^n$  containing only nonnegative components.
\item It is not difficult to see that all such sequences are necessarily infinite.  If such a sequence were finite, we could select the two vectors with the largest $x_1$-component (breaking ties by considering the $x_2$-component, etc.).  The resultant sum of these two vectors would be unique among sums of distinct vectors already in the set, which leads to a contradiction.
\item The definition makes sense for arbitrary vectors equipped with a notion of size.  In general, 
the most interesting dynamics can be expected when $a_1v_1 + a_2 v_2 + \dots + a_kv_k =0$ has nontrivial solutions over $\mathbb{Z}^k$ (\S 2).  However, even an absence of nontrivial solutions gives rise to many complexities (\S 2.4).
\end{enumerate}

\subsection{Outline of the Paper.} We start by presenting several examples in \S 2 and showing how, in many cases, we can establish the existence of regular repeating patterns. \S 3 contains several universality results stating that, for a large number of initial conditions, the arising behavior is universal: in particular, Ulam sets depend only very weakly on the
notion of norm used (despite the definition's emphasis on adding the `smallest' vector with a unique representation). \S 3 also exhibits numerical results suggesting
that the generic case of three initial conditions has surprisingly complex structures and symmetries that we do not rigorously understand. \S 4 is devoted to the column phenomenon: we provide at least a partial understanding of regular periodic structures such as the ones observed to be branching away from the $x-$axis in Figure \ref{fig:ulamsmall} (above).  \S 5 concludes the paper by listing several open problems.

\section{Lattices and Non-Lattices}
\subsection{Lattices} We first consider the case of two vectors $\left\{v_1, v_2\right\} \subset \mathbb{R}_{\geq 0}^2$. If $v_2 = c v_1$, then we merely recreate the one-dimensional classical Ulam sequence with initial conditions
$\left\{1, c\right\}$ where $c > 0$.  If $c = p/q$ is rational, then we can reduce it to the one-dimensional Ulam sequence with initial conditions $\left\{p, q \right\}$. If $c$ is irrational, then we will be able to use
Lemma 2 (\S 3.2) to deduce that the elements in the one-dimensional sequence demonstrate the universal dynamics created by the initial set $\left\{(1,0), (0,1)\right\} \subset \mathbb{R}^2_{\geq 0}$. We now determine the dynamics of such sets.

\begin{center}
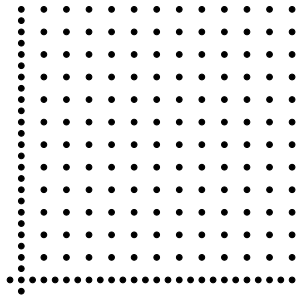
\begin{figure}[h!]
\begin{tikzpicture}[scale=0.15]
\filldraw (1,0) circle (0.25cm);
\filldraw (0,1) circle (0.25cm);
\filldraw (1,1) circle (0.25cm);
\filldraw (1,2) circle (0.25cm);
\filldraw (2,1) circle (0.25cm);
\filldraw (1,3) circle (0.25cm);
\filldraw (3,1) circle (0.25cm);
\filldraw (1,4) circle (0.25cm);
\filldraw (4,1) circle (0.25cm);
\filldraw (1,5) circle (0.25cm);
\filldraw (5,1) circle (0.25cm);
\filldraw (1,6) circle (0.25cm);
\filldraw (6,1) circle (0.25cm);
\filldraw (1,7) circle (0.25cm);
\filldraw (7,1) circle (0.25cm);
\filldraw (1,8) circle (0.25cm);
\filldraw (8,1) circle (0.25cm);
\filldraw (1,9) circle (0.25cm);
\filldraw (9,1) circle (0.25cm);
\filldraw (1,10) circle (0.25cm);
\filldraw (10,1) circle (0.25cm);
\filldraw (1,11) circle (0.25cm);
\filldraw (11,1) circle (0.25cm);
\filldraw (1,12) circle (0.25cm);
\filldraw (12,1) circle (0.25cm);
\filldraw (1,13) circle (0.25cm);
\filldraw (13,1) circle (0.25cm);
\filldraw (1,14) circle (0.25cm);
\filldraw (14,1) circle (0.25cm);
\filldraw (1,15) circle (0.25cm);
\filldraw (15,1) circle (0.25cm);
\filldraw (1,16) circle (0.25cm);
\filldraw (16,1) circle (0.25cm);
\filldraw (1,17) circle (0.25cm);
\filldraw (17,1) circle (0.25cm);
\filldraw (1,18) circle (0.25cm);
\filldraw (18,1) circle (0.25cm);
\filldraw (1,19) circle (0.25cm);
\filldraw (19,1) circle (0.25cm);
\filldraw (1,20) circle (0.25cm);
\filldraw (20,1) circle (0.25cm);
\filldraw (1,21) circle (0.25cm);
\filldraw (21,1) circle (0.25cm);
\filldraw (1,22) circle (0.25cm);
\filldraw (22,1) circle (0.25cm);
\filldraw (1,23) circle (0.25cm);
\filldraw (23,1) circle (0.25cm);
\filldraw (1,24) circle (0.25cm);
\filldraw (24,1) circle (0.25cm);
\filldraw (1,25) circle (0.25cm);
\filldraw (25,1) circle (0.25cm);
\filldraw (3,3) circle (0.25cm);
\filldraw (3,5) circle (0.25cm);
\filldraw (3,7) circle (0.25cm);
\filldraw (3,9) circle (0.25cm);
\filldraw (3,11) circle (0.25cm);
\filldraw (3,13) circle (0.25cm);
\filldraw (3,15) circle (0.25cm);
\filldraw (3,17) circle (0.25cm);
\filldraw (3,19) circle (0.25cm);
\filldraw (3,21) circle (0.25cm);
\filldraw (3,23) circle (0.25cm);
\filldraw (3,25) circle (0.25cm);
\filldraw (5,3) circle (0.25cm);
\filldraw (5,5) circle (0.25cm);
\filldraw (5,7) circle (0.25cm);
\filldraw (5,9) circle (0.25cm);
\filldraw (5,11) circle (0.25cm);
\filldraw (5,13) circle (0.25cm);
\filldraw (5,15) circle (0.25cm);
\filldraw (5,17) circle (0.25cm);
\filldraw (5,19) circle (0.25cm);
\filldraw (5,21) circle (0.25cm);
\filldraw (5,23) circle (0.25cm);
\filldraw (5,25) circle (0.25cm);
\filldraw (7,3) circle (0.25cm);
\filldraw (7,5) circle (0.25cm);
\filldraw (7,7) circle (0.25cm);
\filldraw (7,9) circle (0.25cm);
\filldraw (7,11) circle (0.25cm);
\filldraw (7,13) circle (0.25cm);
\filldraw (7,15) circle (0.25cm);
\filldraw (7,17) circle (0.25cm);
\filldraw (7,19) circle (0.25cm);
\filldraw (7,21) circle (0.25cm);
\filldraw (7,23) circle (0.25cm);
\filldraw (7,25) circle (0.25cm);
\filldraw (9,3) circle (0.25cm);
\filldraw (9,5) circle (0.25cm);
\filldraw (9,7) circle (0.25cm);
\filldraw (9,9) circle (0.25cm);
\filldraw (9,11) circle (0.25cm);
\filldraw (9,13) circle (0.25cm);
\filldraw (9,15) circle (0.25cm);
\filldraw (9,17) circle (0.25cm);
\filldraw (9,19) circle (0.25cm);
\filldraw (9,21) circle (0.25cm);
\filldraw (9,23) circle (0.25cm);
\filldraw (9,25) circle (0.25cm);
\filldraw (11,3) circle (0.25cm);
\filldraw (11,5) circle (0.25cm);
\filldraw (11,7) circle (0.25cm);
\filldraw (11,9) circle (0.25cm);
\filldraw (11,11) circle (0.25cm);
\filldraw (11,13) circle (0.25cm);
\filldraw (11,15) circle (0.25cm);
\filldraw (11,17) circle (0.25cm);
\filldraw (11,19) circle (0.25cm);
\filldraw (11,21) circle (0.25cm);
\filldraw (11,23) circle (0.25cm);
\filldraw (11,25) circle (0.25cm);
\filldraw (13,3) circle (0.25cm);
\filldraw (13,5) circle (0.25cm);
\filldraw (13,7) circle (0.25cm);
\filldraw (13,9) circle (0.25cm);
\filldraw (13,11) circle (0.25cm);
\filldraw (13,13) circle (0.25cm);
\filldraw (13,15) circle (0.25cm);
\filldraw (13,17) circle (0.25cm);
\filldraw (13,19) circle (0.25cm);
\filldraw (13,21) circle (0.25cm);
\filldraw (13,23) circle (0.25cm);
\filldraw (13,25) circle (0.25cm);
\filldraw (15,3) circle (0.25cm);
\filldraw (15,5) circle (0.25cm);
\filldraw (15,7) circle (0.25cm);
\filldraw (15,9) circle (0.25cm);
\filldraw (15,11) circle (0.25cm);
\filldraw (15,13) circle (0.25cm);
\filldraw (15,15) circle (0.25cm);
\filldraw (15,17) circle (0.25cm);
\filldraw (15,19) circle (0.25cm);
\filldraw (15,21) circle (0.25cm);
\filldraw (15,23) circle (0.25cm);
\filldraw (15,25) circle (0.25cm);
\filldraw (17,3) circle (0.25cm);
\filldraw (17,5) circle (0.25cm);
\filldraw (17,7) circle (0.25cm);
\filldraw (17,9) circle (0.25cm);
\filldraw (17,11) circle (0.25cm);
\filldraw (17,13) circle (0.25cm);
\filldraw (17,15) circle (0.25cm);
\filldraw (17,17) circle (0.25cm);
\filldraw (17,19) circle (0.25cm);
\filldraw (17,21) circle (0.25cm);
\filldraw (17,23) circle (0.25cm);
\filldraw (17,25) circle (0.25cm);
\filldraw (19,3) circle (0.25cm);
\filldraw (19,5) circle (0.25cm);
\filldraw (19,7) circle (0.25cm);
\filldraw (19,9) circle (0.25cm);
\filldraw (19,11) circle (0.25cm);
\filldraw (19,13) circle (0.25cm);
\filldraw (19,15) circle (0.25cm);
\filldraw (19,17) circle (0.25cm);
\filldraw (19,19) circle (0.25cm);
\filldraw (19,21) circle (0.25cm);
\filldraw (19,23) circle (0.25cm);
\filldraw (19,25) circle (0.25cm);
\filldraw (21,3) circle (0.25cm);
\filldraw (21,5) circle (0.25cm);
\filldraw (21,7) circle (0.25cm);
\filldraw (21,9) circle (0.25cm);
\filldraw (21,11) circle (0.25cm);
\filldraw (21,13) circle (0.25cm);
\filldraw (21,15) circle (0.25cm);
\filldraw (21,17) circle (0.25cm);
\filldraw (21,19) circle (0.25cm);
\filldraw (21,21) circle (0.25cm);
\filldraw (21,23) circle (0.25cm);
\filldraw (21,25) circle (0.25cm);
\filldraw (23,3) circle (0.25cm);
\filldraw (23,5) circle (0.25cm);
\filldraw (23,7) circle (0.25cm);
\filldraw (23,9) circle (0.25cm);
\filldraw (23,11) circle (0.25cm);
\filldraw (23,13) circle (0.25cm);
\filldraw (23,15) circle (0.25cm);
\filldraw (23,17) circle (0.25cm);
\filldraw (23,19) circle (0.25cm);
\filldraw (23,21) circle (0.25cm);
\filldraw (23,23) circle (0.25cm);
\filldraw (23,25) circle (0.25cm);
\filldraw (25,3) circle (0.25cm);
\filldraw (25,5) circle (0.25cm);
\filldraw (25,7) circle (0.25cm);
\filldraw (25,9) circle (0.25cm);
\filldraw (25,11) circle (0.25cm);
\filldraw (25,13) circle (0.25cm);
\filldraw (25,15) circle (0.25cm);
\filldraw (25,17) circle (0.25cm);
\filldraw (25,19) circle (0.25cm);
\filldraw (25,21) circle (0.25cm);
\filldraw (25,23) circle (0.25cm);
\filldraw (25,25) circle (0.25cm);
\end{tikzpicture}
\captionsetup{width = 0.95\textwidth}
\caption{The set arising from $\left\{(1,0), (0,1)\right\}$.}
\label{fig:twodlattice}
\end{figure}
\end{center}
\vspace{-20pt}

\begin{theorem}[Lattice in Two Dimensions]
The set arising from nonparallel $\left\{v_1, v_2\right\} \subset \mathbb{R}_{\geq 0}^2$ consists of all vectors of the form $v_1 +n v_2$ and $n v_1 +v_2$ for $n \in \mathbb{N}$ and all vectors of the form $m v_1 +n v_2$ with $m, n \geq 3$ both odd integers.  
\end{theorem}
\begin{proof}
All terms in the sequence are contained in the set
$$\left\{ k_1v_1+k_2v_2: k_1, k_2 \in \mathbb{N} \right\}.$$ 
Since $v_1$ and $v_2$ are linearly independent, the above representation is unique. Whether $mv_1+nv_2$  is contained
in the set can be determined by knowing which elements of the form
$$\left\{ k_1v_1+k_2v_2: 0 \leq k_1 \leq m \wedge 0 \leq k_2 \leq n  \right\}$$ 
are contained in the Ulam set.  First, $v_1+v_2$ is uniquely representable and is thus in the Ulam set.  We can now see inductively that all elements 
$$ \left\{v_1+nv_2 : n \in \mathbb{N} \right\} \qquad \mbox{and} \qquad  \left\{nv_1+v_2 : n \in \mathbb{N} \right\} $$
are uniquely representable and hence are contained in the set. We will now show by induction that for $m,n \geq 2$, the vector $mv_1+nv_2$ is an element iff both $m$ and $n$ are odd.  The base case is $m=2$ or $n=2$.  We have 
$$2v_1 + 2v_2=(2v_1+v_2)+v_2=(v_1+2v_2)+v_1,$$ 
implying that it is not included.  By the same token, for $n\geq 3$, 
$$2v_1+nv_2=(v_1+v_2)+(v_1+(n-1)v_2)=(v_1+nv_2)+v_1$$  
is is not unique either.
 Symmetrically, vectors of the form $nv_1+2v_2$ are excluded for the same reason.  We now consider vectors $mv_1+nv_2$ with $m,n \geq 3$ and note that we always have a representation of the type
$$mv_1+nv_2=(v_1+(n-1)v_2)+((m-1)v_1+v_2),$$
using vectors already established to be in the set. Depending on the parity of $m$ and $n$, we can now distinguish four cases.
If $m$ and $n$ are both even, then we get a second representation 
$$mv_1+nv_2=((m-1)v_1+(n-1)v_2)+(v_1+v_2),$$
where $(m-1)v_1+(n-1)v_2$ is contained in the set by the inductive hypothesis.
If $m$ is even and $n$ is odd, then
$$mv_1+nv_2=((m-1)v_1+nv_2)+v_1$$
is a second representation, and the case of $m$ odd, $n$ even follows by symmetry.
It remains to show that there is no second representation when both $m$ and $n$ are both odd.  Let $S_1$ denote the set of vectors of the form $v_1+nv_2$ or $nv_1+v_2$, and let $S_2$ denote all the other vectors in the Ulam set with $x-$coordinate at most $m$ and $y-$coordinate at most $n$.  Note that, by hypothesis, the coefficients of $v_1$ and $v_2$ are both odd for all elements of $S_2$.  The fact that $m,n \geq 3$ means that $mv_1+nv_2$ can be written uniquely as the sum of 2 elements of $S_1$: this is the representation we found before beginning casework.  The sum of any 2 elements of $S_2$ has even coefficients for $v_1$ and $v_2$, so $mv_1+nv_2$ cannot be expressed in this way.  Similarly, the sum of an element of $S_1$ and an element of $S_2$ must have at least 1 even coefficient (from the sum of the 1 in the $S_1$ element and an odd coefficient in the $S_2$ element).  This exhausts all possibilities.
\end{proof}

\subsection{Special Cases with Regular Behavior.} The purpose of this section is to demonstrate that a variety of cases can actually be rigorously dealt with; the proofs rely mainly on using the right type of induction and are only sketched.

\subsubsection{$\left\{(2,0), (0,1), (3,1)\right\}$}  This sequence consists of exactly the points $(2,0)$ and $(0,1)$, and all points of the form $(n,1), (2,n)$, and $(3,n)$ where $n \geq 2$.

\begin{center}
\begin{figure}[h!]
\begin{tikzpicture}[scale=0.12]

\filldraw (2,0) circle (0.2cm);
\filldraw (0,1) circle (0.2cm);
\filldraw (2,2) circle (0.2cm);
\filldraw (3,2) circle (0.2cm);
\filldraw (2,1) circle (0.2cm);
\filldraw (2,3) circle (0.2cm);
\filldraw (3,3) circle (0.2cm);
\filldraw (3,1) circle (0.2cm);
\filldraw (2,4) circle (0.2cm);
\filldraw (3,4) circle (0.2cm);
\filldraw (4,1) circle (0.2cm);
\filldraw (2,5) circle (0.2cm);
\filldraw (3,5) circle (0.2cm);
\filldraw (5,1) circle (0.2cm);
\filldraw (2,6) circle (0.2cm);
\filldraw (3,6) circle (0.2cm);
\filldraw (6,1) circle (0.2cm);
\filldraw (2,7) circle (0.2cm);
\filldraw (3,7) circle (0.2cm);
\filldraw (7,1) circle (0.2cm);
\filldraw (2,8) circle (0.2cm);
\filldraw (3,8) circle (0.2cm);
\filldraw (8,1) circle (0.2cm);
\filldraw (2,9) circle (0.2cm);
\filldraw (3,9) circle (0.2cm);
\filldraw (9,1) circle (0.2cm);
\filldraw (2,10) circle (0.2cm);
\filldraw (3,10) circle (0.2cm);
\filldraw (10,1) circle (0.2cm);
\filldraw (2,11) circle (0.2cm);
\filldraw (3,11) circle (0.2cm);
\filldraw (11,1) circle (0.2cm);
\filldraw (2,12) circle (0.2cm);
\filldraw (3,12) circle (0.2cm);
\filldraw (12,1) circle (0.2cm);
\filldraw (2,13) circle (0.2cm);
\filldraw (3,13) circle (0.2cm);
\filldraw (13,1) circle (0.2cm);
\filldraw (2,14) circle (0.2cm);
\filldraw (3,14) circle (0.2cm);
\filldraw (14,1) circle (0.2cm);
\filldraw (2,15) circle (0.2cm);
\filldraw (3,15) circle (0.2cm);
\filldraw (15,1) circle (0.2cm);
\filldraw (2,16) circle (0.2cm);
\filldraw (3,16) circle (0.2cm);
\filldraw (16,1) circle (0.2cm);
\filldraw (2,17) circle (0.2cm);
\filldraw (3,17) circle (0.2cm);
\filldraw (17,1) circle (0.2cm);
\filldraw (2,18) circle (0.2cm);
\filldraw (3,18) circle (0.2cm);
\filldraw (18,1) circle (0.2cm);
\filldraw (2,19) circle (0.2cm);
\filldraw (3,19) circle (0.2cm);
\filldraw (19,1) circle (0.2cm);
\filldraw (2,20) circle (0.2cm);
\filldraw (3,20) circle (0.2cm);
\filldraw (20,1) circle (0.2cm);
\filldraw (2,21) circle (0.2cm);
\filldraw (3,21) circle (0.2cm);
\filldraw (21,1) circle (0.2cm);
\filldraw (2,22) circle (0.2cm);
\filldraw (3,22) circle (0.2cm);
\filldraw (22,1) circle (0.2cm);
\filldraw (2,23) circle (0.2cm);
\filldraw (3,23) circle (0.2cm);
\filldraw (23,1) circle (0.2cm);
\filldraw (2,24) circle (0.2cm);
\filldraw (3,24) circle (0.2cm);
\filldraw (24,1) circle (0.2cm);
\filldraw (2,25) circle (0.2cm);
\filldraw (3,25) circle (0.2cm);
\filldraw (25,1) circle (0.2cm);
\filldraw (2,26) circle (0.2cm);
\filldraw (3,26) circle (0.2cm);
\filldraw (26,1) circle (0.2cm);
\filldraw (2,27) circle (0.2cm);
\filldraw (3,27) circle (0.2cm);
\filldraw (27,1) circle (0.2cm);
\filldraw (2,28) circle (0.2cm);
\filldraw (3,28) circle (0.2cm);
\filldraw (28,1) circle (0.2cm);
\filldraw (2,29) circle (0.2cm);
\filldraw (3,29) circle (0.2cm);
\filldraw (29,1) circle (0.2cm);
\end{tikzpicture}
\captionsetup{width = 0.8\textwidth}
\caption{The set arising from $\left\{(2,0), (0,1), (3,1)\right\}$.}
\label{fig:twodlattice}
\end{figure}
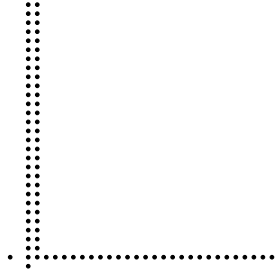
\end{center}
\vspace{-20pt}

\begin{proof}
 We start by noting that there can be no point other than $(2,0)$ on the $x-$axis since $(2,0)$ is the only point with $y-$component 0.  The same argument tells us that $(0,1)$ is the only point on the $y-$axis, and, by the same token, there can never be any point with $x-$component 1.  Now for $n \geq 2$, we have $(n,1)=(n-2,1)+(2,0)$ uniquely.  (This is simple $n \rightarrow n+2$ induction with base cases $(2,1)$ and $(3,1)$.)  Inductively, we also have $(2,n)=(2,n-1)+(0,1)$ uniquely for $n \geq 2$, so all points of the form $(2,n)$ are included.  Similarly, $(3,n)=(3,n-1)+(0,1)$ uniquely.  Note that, again, the absence of any element of the Ulam set with $x-$component 1 makes uniqueness easy to see.  Next, we have $(4,2)=(2,2)+(2,0)=(4,1)+(0,1)$ not uniquely.  And for $n \geq 5$, $(n,2)=(n,1)+(0,1)=(n-2,1)+(2,1)$ is also excluded.  For $n \geq 4$, we have $(4,n)=(2,0)+(2,n)=(2,1)+(2,n-1)$ not uniquely.  Finally, for any remaining `interior point' with $m \geq 5$, $n \geq 3$, we have that $(m,n)=(m-2,1)+(2,n-1)=(m-3,1)+(3,n-1)$ is not unique and is thus excluded.
\end{proof}

\subsubsection{$\left\{(1,0), (0,1), (2,3)\right\}$}  This sequence consists of all points of the form $(n,1)$ and $(1,n)$ where $n \in \mathbb{N}$, the point $(2,3)$, and all points of the form $(2n+4,2m+3)$ for $m,n \in \mathbb{N}.$  The proof proceeds using parity distinctions as in Theorem 1.

\begin{center}
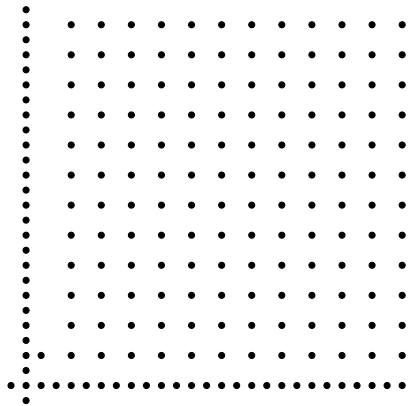
\begin{figure}[h!]
\begin{tikzpicture}[scale=0.1]
\begin{tikzpicture}[scale=0.2]

\filldraw (1,0) circle (0.2cm);
\filldraw (0,1) circle (0.2cm);
\filldraw (1,1) circle (0.2cm);
\filldraw (2,3) circle (0.2cm);
\filldraw (2,1) circle (0.2cm);
\filldraw (1,2) circle (0.2cm);
\filldraw (3,1) circle (0.2cm);
\filldraw (1,3) circle (0.2cm);
\filldraw (4,1) circle (0.2cm);
\filldraw (1,4) circle (0.2cm);
\filldraw (5,1) circle (0.2cm);
\filldraw (1,5) circle (0.2cm);
\filldraw (6,1) circle (0.2cm);
\filldraw (1,6) circle (0.2cm);
\filldraw (7,1) circle (0.2cm);
\filldraw (1,7) circle (0.2cm);
\filldraw (8,1) circle (0.2cm);
\filldraw (1,8) circle (0.2cm);
\filldraw (9,1) circle (0.2cm);
\filldraw (1,9) circle (0.2cm);
\filldraw (10,1) circle (0.2cm);
\filldraw (1,10) circle (0.2cm);
\filldraw (11,1) circle (0.2cm);
\filldraw (1,11) circle (0.2cm);
\filldraw (12,1) circle (0.2cm);
\filldraw (1,12) circle (0.2cm);
\filldraw (13,1) circle (0.2cm);
\filldraw (1,13) circle (0.2cm);
\filldraw (14,1) circle (0.2cm);
\filldraw (1,14) circle (0.2cm);
\filldraw (15,1) circle (0.2cm);
\filldraw (1,15) circle (0.2cm);
\filldraw (16,1) circle (0.2cm);
\filldraw (1,16) circle (0.2cm);
\filldraw (17,1) circle (0.2cm);
\filldraw (1,17) circle (0.2cm);
\filldraw (18,1) circle (0.2cm);
\filldraw (1,18) circle (0.2cm);
\filldraw (19,1) circle (0.2cm);
\filldraw (1,19) circle (0.2cm);
\filldraw (20,1) circle (0.2cm);
\filldraw (1,20) circle (0.2cm);
\filldraw (21,1) circle (0.2cm);
\filldraw (1,21) circle (0.2cm);
\filldraw (22,1) circle (0.2cm);
\filldraw (1,22) circle (0.2cm);
\filldraw (23,1) circle (0.2cm);
\filldraw (1,23) circle (0.2cm);
\filldraw (24,1) circle (0.2cm);
\filldraw (1,24) circle (0.2cm);
\filldraw (25,1) circle (0.2cm);
\filldraw (1,25) circle (0.2cm);
\filldraw (26,1) circle (0.2cm);
\filldraw (1,26) circle (0.2cm);
\filldraw (4,3) circle (0.2cm);
\filldraw (4,5) circle (0.2cm);
\filldraw (4,7) circle (0.2cm);
\filldraw (4,9) circle (0.2cm);
\filldraw (4,11) circle (0.2cm);
\filldraw (4,13) circle (0.2cm);
\filldraw (4,15) circle (0.2cm);
\filldraw (4,17) circle (0.2cm);
\filldraw (4,19) circle (0.2cm);
\filldraw (4,21) circle (0.2cm);
\filldraw (4,23) circle (0.2cm);
\filldraw (4,25) circle (0.2cm);
\filldraw (6,3) circle (0.2cm);
\filldraw (6,5) circle (0.2cm);
\filldraw (6,7) circle (0.2cm);
\filldraw (6,9) circle (0.2cm);
\filldraw (6,11) circle (0.2cm);
\filldraw (6,13) circle (0.2cm);
\filldraw (6,15) circle (0.2cm);
\filldraw (6,17) circle (0.2cm);
\filldraw (6,19) circle (0.2cm);
\filldraw (6,21) circle (0.2cm);
\filldraw (6,23) circle (0.2cm);
\filldraw (6,25) circle (0.2cm);
\filldraw (8,3) circle (0.2cm);
\filldraw (8,5) circle (0.2cm);
\filldraw (8,7) circle (0.2cm);
\filldraw (8,9) circle (0.2cm);
\filldraw (8,11) circle (0.2cm);
\filldraw (8,13) circle (0.2cm);
\filldraw (8,15) circle (0.2cm);
\filldraw (8,17) circle (0.2cm);
\filldraw (8,19) circle (0.2cm);
\filldraw (8,21) circle (0.2cm);
\filldraw (8,23) circle (0.2cm);
\filldraw (8,25) circle (0.2cm);
\filldraw (10,3) circle (0.2cm);
\filldraw (10,5) circle (0.2cm);
\filldraw (10,7) circle (0.2cm);
\filldraw (10,9) circle (0.2cm);
\filldraw (10,11) circle (0.2cm);
\filldraw (10,13) circle (0.2cm);
\filldraw (10,15) circle (0.2cm);
\filldraw (10,17) circle (0.2cm);
\filldraw (10,19) circle (0.2cm);
\filldraw (10,21) circle (0.2cm);
\filldraw (10,23) circle (0.2cm);
\filldraw (10,25) circle (0.2cm);
\filldraw (12,3) circle (0.2cm);
\filldraw (12,5) circle (0.2cm);
\filldraw (12,7) circle (0.2cm);
\filldraw (12,9) circle (0.2cm);
\filldraw (12,11) circle (0.2cm);
\filldraw (12,13) circle (0.2cm);
\filldraw (12,15) circle (0.2cm);
\filldraw (12,17) circle (0.2cm);
\filldraw (12,19) circle (0.2cm);
\filldraw (12,21) circle (0.2cm);
\filldraw (12,23) circle (0.2cm);
\filldraw (12,25) circle (0.2cm);
\filldraw (14,3) circle (0.2cm);
\filldraw (14,5) circle (0.2cm);
\filldraw (14,7) circle (0.2cm);
\filldraw (14,9) circle (0.2cm);
\filldraw (14,11) circle (0.2cm);
\filldraw (14,13) circle (0.2cm);
\filldraw (14,15) circle (0.2cm);
\filldraw (14,17) circle (0.2cm);
\filldraw (14,19) circle (0.2cm);
\filldraw (14,21) circle (0.2cm);
\filldraw (14,23) circle (0.2cm);
\filldraw (14,25) circle (0.2cm);
\filldraw (16,3) circle (0.2cm);
\filldraw (16,5) circle (0.2cm);
\filldraw (16,7) circle (0.2cm);
\filldraw (16,9) circle (0.2cm);
\filldraw (16,11) circle (0.2cm);
\filldraw (16,13) circle (0.2cm);
\filldraw (16,15) circle (0.2cm);
\filldraw (16,17) circle (0.2cm);
\filldraw (16,19) circle (0.2cm);
\filldraw (16,21) circle (0.2cm);
\filldraw (16,23) circle (0.2cm);
\filldraw (16,25) circle (0.2cm);
\filldraw (18,3) circle (0.2cm);
\filldraw (18,5) circle (0.2cm);
\filldraw (18,7) circle (0.2cm);
\filldraw (18,9) circle (0.2cm);
\filldraw (18,11) circle (0.2cm);
\filldraw (18,13) circle (0.2cm);
\filldraw (18,15) circle (0.2cm);
\filldraw (18,17) circle (0.2cm);
\filldraw (18,19) circle (0.2cm);
\filldraw (18,21) circle (0.2cm);
\filldraw (18,23) circle (0.2cm);
\filldraw (18,25) circle (0.2cm);
\filldraw (20,3) circle (0.2cm);
\filldraw (20,5) circle (0.2cm);
\filldraw (20,7) circle (0.2cm);
\filldraw (20,9) circle (0.2cm);
\filldraw (20,11) circle (0.2cm);
\filldraw (20,13) circle (0.2cm);
\filldraw (20,15) circle (0.2cm);
\filldraw (20,17) circle (0.2cm);
\filldraw (20,19) circle (0.2cm);
\filldraw (20,21) circle (0.2cm);
\filldraw (20,23) circle (0.2cm);
\filldraw (20,25) circle (0.2cm);
\filldraw (22,3) circle (0.2cm);
\filldraw (22,5) circle (0.2cm);
\filldraw (22,7) circle (0.2cm);
\filldraw (22,9) circle (0.2cm);
\filldraw (22,11) circle (0.2cm);
\filldraw (22,13) circle (0.2cm);
\filldraw (22,15) circle (0.2cm);
\filldraw (22,17) circle (0.2cm);
\filldraw (22,19) circle (0.2cm);
\filldraw (22,21) circle (0.2cm);
\filldraw (22,23) circle (0.2cm);
\filldraw (22,25) circle (0.2cm);
\filldraw (24,3) circle (0.2cm);
\filldraw (24,5) circle (0.2cm);
\filldraw (24,7) circle (0.2cm);
\filldraw (24,9) circle (0.2cm);
\filldraw (24,11) circle (0.2cm);
\filldraw (24,13) circle (0.2cm);
\filldraw (24,15) circle (0.2cm);
\filldraw (24,17) circle (0.2cm);
\filldraw (24,19) circle (0.2cm);
\filldraw (24,21) circle (0.2cm);
\filldraw (24,23) circle (0.2cm);
\filldraw (24,25) circle (0.2cm);
\filldraw (26,3) circle (0.2cm);
\filldraw (26,5) circle (0.2cm);
\filldraw (26,7) circle (0.2cm);
\filldraw (26,9) circle (0.2cm);
\filldraw (26,11) circle (0.2cm);
\filldraw (26,13) circle (0.2cm);
\filldraw (26,15) circle (0.2cm);
\filldraw (26,17) circle (0.2cm);
\filldraw (26,19) circle (0.2cm);
\filldraw (26,21) circle (0.2cm);
\filldraw (26,23) circle (0.2cm);
\filldraw (26,25) circle (0.2cm);

\end{tikzpicture}

\end{tikzpicture}
\captionsetup{width = 0.95\textwidth}
\caption{The set arising from $\left\{(1,0), (0,1), (2,3)\right\}$.}
\label{fig:twodlattice}
\end{figure}
\end{center}
\vspace{-20pt}

\subsubsection{$\left\{(3,0), (0,1), (1,1)\right\}$}  This sequence splits $\mathbb{R}^2_{\geq 0}$ in a stable way into two regions  with different lattice behavior. We did not attempt
to obtain a complete proof because, at this point, it seems to require an enormous amount of casework.  However, we do not see a fundamental obstruction to obtaining a proof using the methods employed above.  (Perhaps unsurprisingly, this sort of statement is fairly easy to prove via induction once the correct induction hypothesis is found). (Note added in revision: some of these types of behavior have now been
rigorously classified by Hinman, Kuca, Schlesinger \& Sheydvasser \cite{hinman}.)
\begin{center}
\begin{figure}[h!]
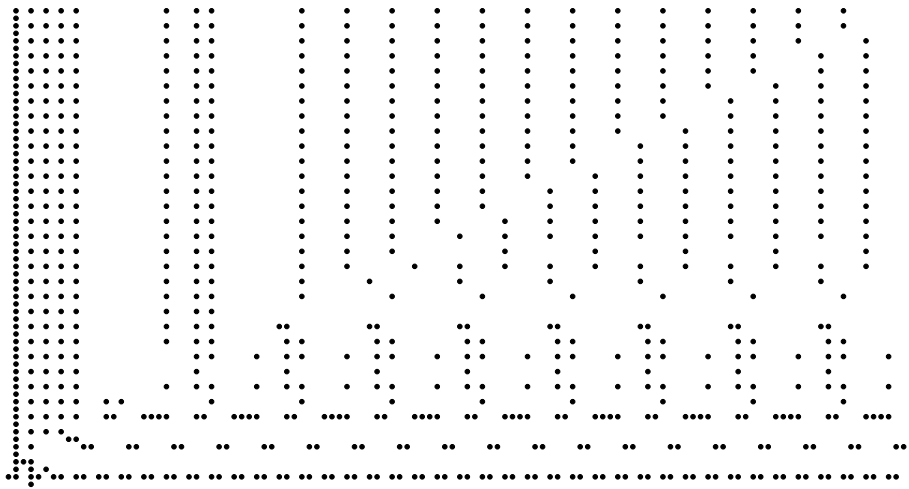


\captionsetup{width = 0.95\textwidth}
\caption{The set arising from $\left\{(3,0), (0,1), (1,1)\right\}$ splits the domain into two lattices.}
\label{fig:twodlattice}
\end{figure}
\end{center}
\vspace{-10pt}

\subsection{Linear Transformations}
There is a useful invariance under certain linear transformations which we actively exploit in the study of three initial vectors $\left\{v_1, v_2, v_3\right\} \subset \mathbb{R}_{\geq 0}^2$.
\begin{lemma} Let $\left\{ v_1, \dots, v_k\right\} \subset \mathbb{R}^2_{ \geq 0}$ span $\mathbb{R}^2$. Then there exists an invertible linear transformation $T:\mathbb{R}^2 \rightarrow \mathbb{R}^2$ that maps $\left\{ v_1, \dots, v_k\right\} \subset \mathbb{R}^2_{ \geq 0}$ to another set in $\mathbb{R}^2_{\geq 0}$ such that at least one of the transformed vectors lies on the $x-$axis and at least one on the $y-$axis.  Moreover, the Ulam sets arising from the original and transformed sets of initial vectors are structurally equivalent (in the obvious sense).
\end{lemma}
\begin{proof} The proof is fairly simple: it is easy to see that sets are invariant under small rotations that keep all the vectors in the positive quadrant $\mathbb{R}^2_{\geq 0}$. 
This allows us to map the
vector(s) with the smallest slope to the $x-$axis. Finally, it is also easy to see that Ulam sets of this type are invariant under shear transformations of the form
$$ S:(x,y) \rightarrow (x - cy, y),$$
and the result follows from composition.
\end{proof}


This property is quite useful when studying the case of three initial vectors in $\mathbb{R}^2_{\geq 0}$.  Invariance under certain types of linear
transformation easily generalizes to higher dimensions and should be a valuable symmetry in the systematic investigation of these sets.  We also remark that if all the initial vectors are contained in $\mathbb{Z}_{\geq 0}^2$, then the transformed vectors will all be in $\mathbb{Q}^2_{\geq 0}$, and scaling them yields another set in $\mathbb{Z}_{\geq 0}^2$.  We give a more formal and thorough treatment of structural equivalence in \S3.

\subsection{Unit Vectors in Three Dimensions}
A natural topic of inquiry is the case of the three canonical unit vectors in $\mathbb{R}^3$ because their behavior is the universal behavior for three vectors $\left\{v_1, v_2, v_3\right\}$ that are linearly independent over $\mathbb{Z}$ (see \S 3.2).
\begin{figure}[h!]
\begin{minipage}[b]{0.48\linewidth}
\centering
\includegraphics[width=0.8\textwidth]{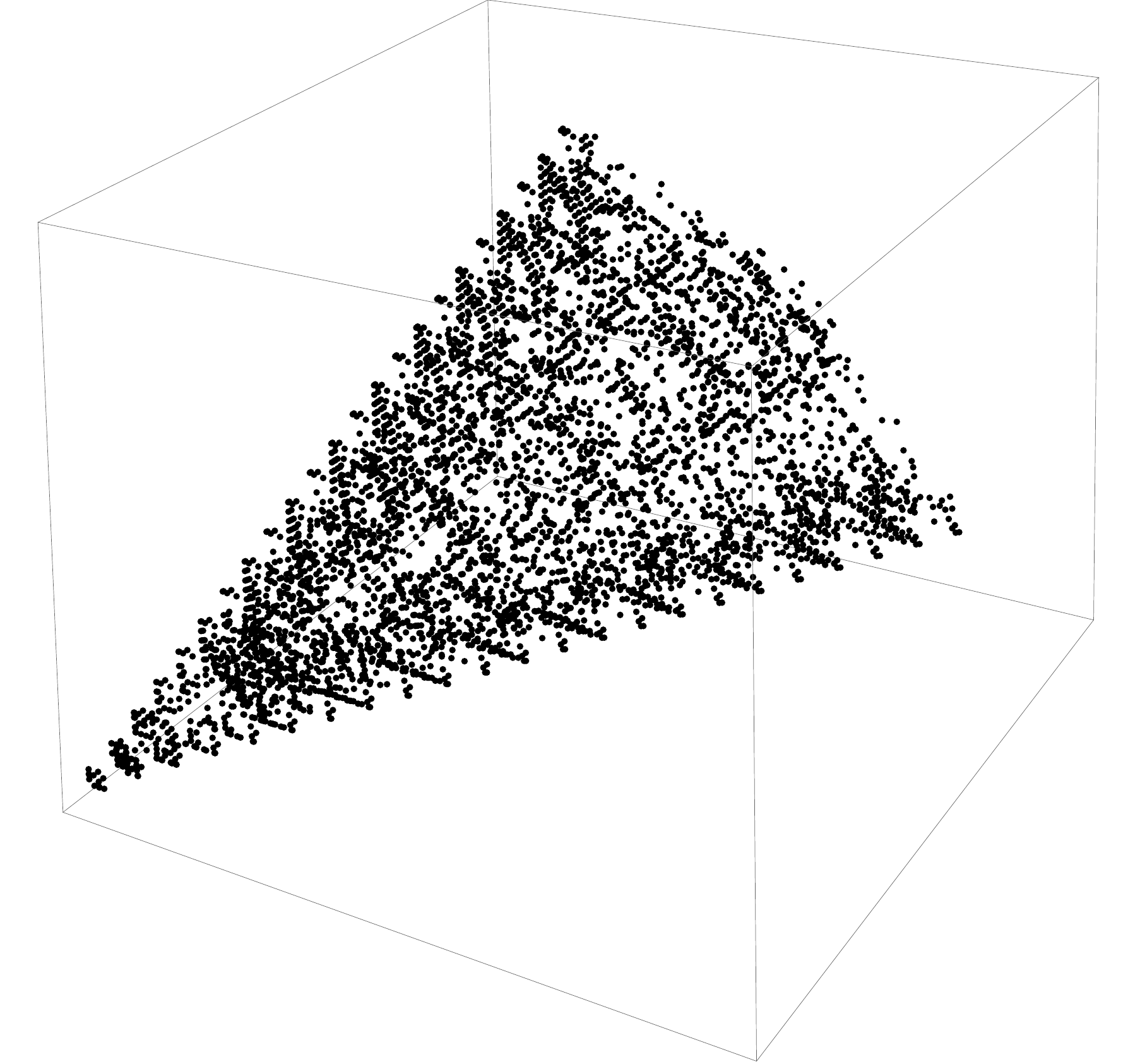}
\end{minipage}
\begin{minipage}[b]{0.48\linewidth}
\centering
\includegraphics[width=0.8\textwidth]{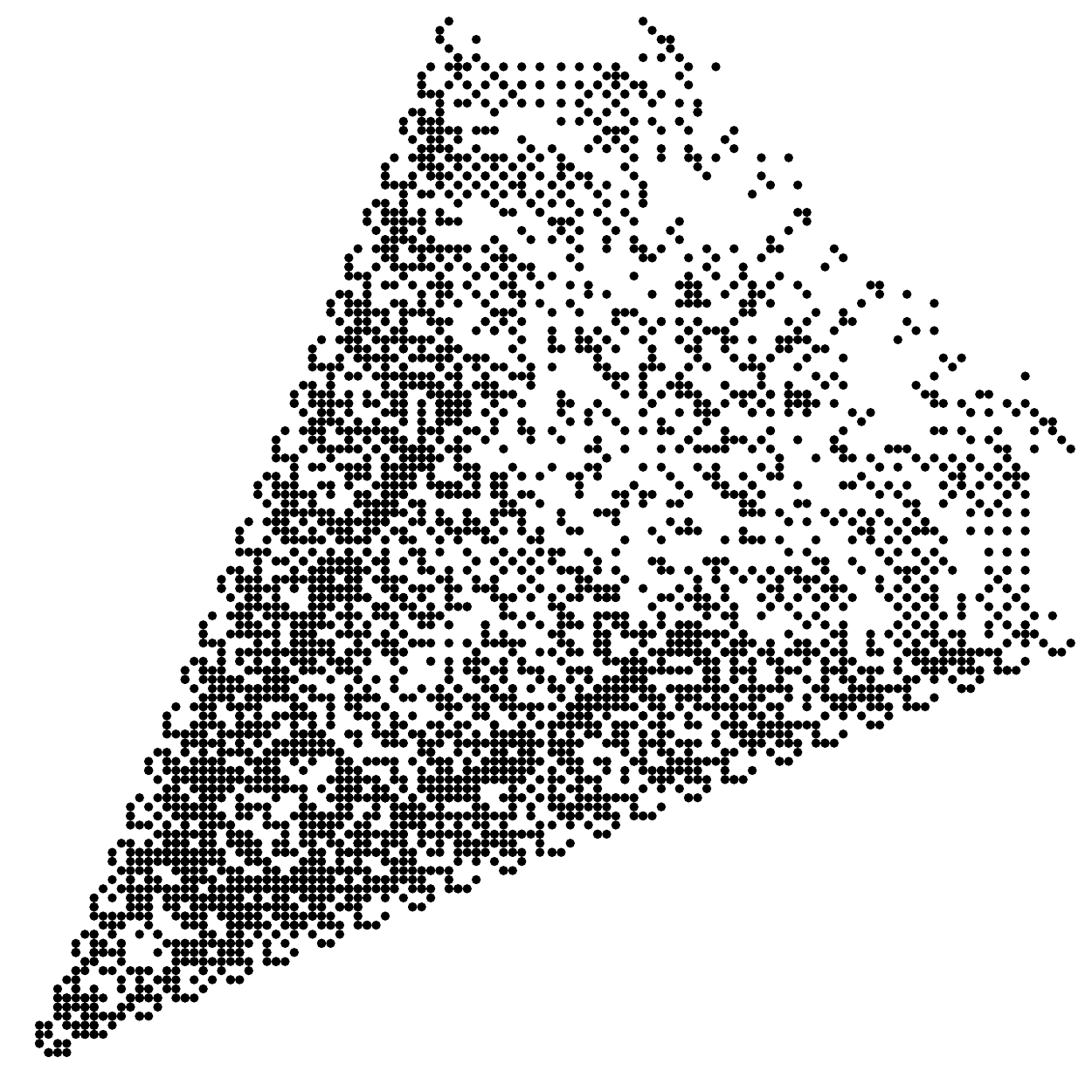}
\end{minipage}
\captionsetup{width = 0.95\textwidth}
\caption{The first few points with all coordinates larger than 2 in the Ulam set generated by the unit vectors of $\mathbb{R}^3$ (left) and a projection of the same points onto the $xy-$plane (right).}
\label{fig:block}
\end{figure}
Computation using the initial set $\{(1,0,0), (0,1,0), (0,0,1)\}$ reveals this structure to be highly nontrivial.  The coordinate planes naturally contain 2-dimensional lattices generated pairwise by the initial vectors characterized above. There is a second type of regular structure appearing in the hyperplane $\left\{(x,y,z) \in \mathbb{R}^3: x = 2\right\}$ (and, by symmetry, the other two hyperplanes) because it contains the points $(2,0,1), (2,1,0)$, and $(2,2m+3, 2n+3)$ for $m,n \in \mathbb{N}$. However, beyond these two planes (or, a total of six hyperplanes with symmetry), there
is a secondary structure unfolding in the interior that seems to have a roughly hexagonal shape and to be centered in the direction $(1,1,1)$ (as is to be expected because of symmetry under permutation of the coordinates).
Simple numerical experiments reveal that the point $(4,6,10)$ (along with the five other points arising from permutation) seems to make the largest angle with the vector $(1,1,1)$; the
second largest angle comes from $(94, 136, 230)$. Our only rigorous result (besides the behavior
of the hyperplanes and the obvious hexagonal symmetry) is that $(n,n,n)$ is never in the set.

\begin{proposition} The Ulam set arising from the canonical basis vectors in $\mathbb{R}^3$ does not contain any vector of the form $(n,n,n) \in \mathbb{N}^3$.
\end{proposition}
\begin{proof} Suppose the statement is false and $(n,n,n) \in \mathbb{N}^3$ is the smallest such element in the set. (Observe that $(0,0,0)$ is not included in the set.) Then $(n,n,n) = (a_1, a_2, a_3) + (b_1,b_2,b_3)$
uniquely for some elements $(a_1, a_2, a_3)\neq (b_1,b_2,b_3)$ in the set. It is clear that not all entries of these elements can be identical because $(n,n,n)$ is the smallest
element in the set with that property. Then, however, at least one of the six possible permutations of the coordinates yields a second representation, and the resulting non-uniqueness  implies
the result.
\end{proof}

\begin{center}
\begin{figure}[h!]
\includegraphics[width=0.45\textwidth]{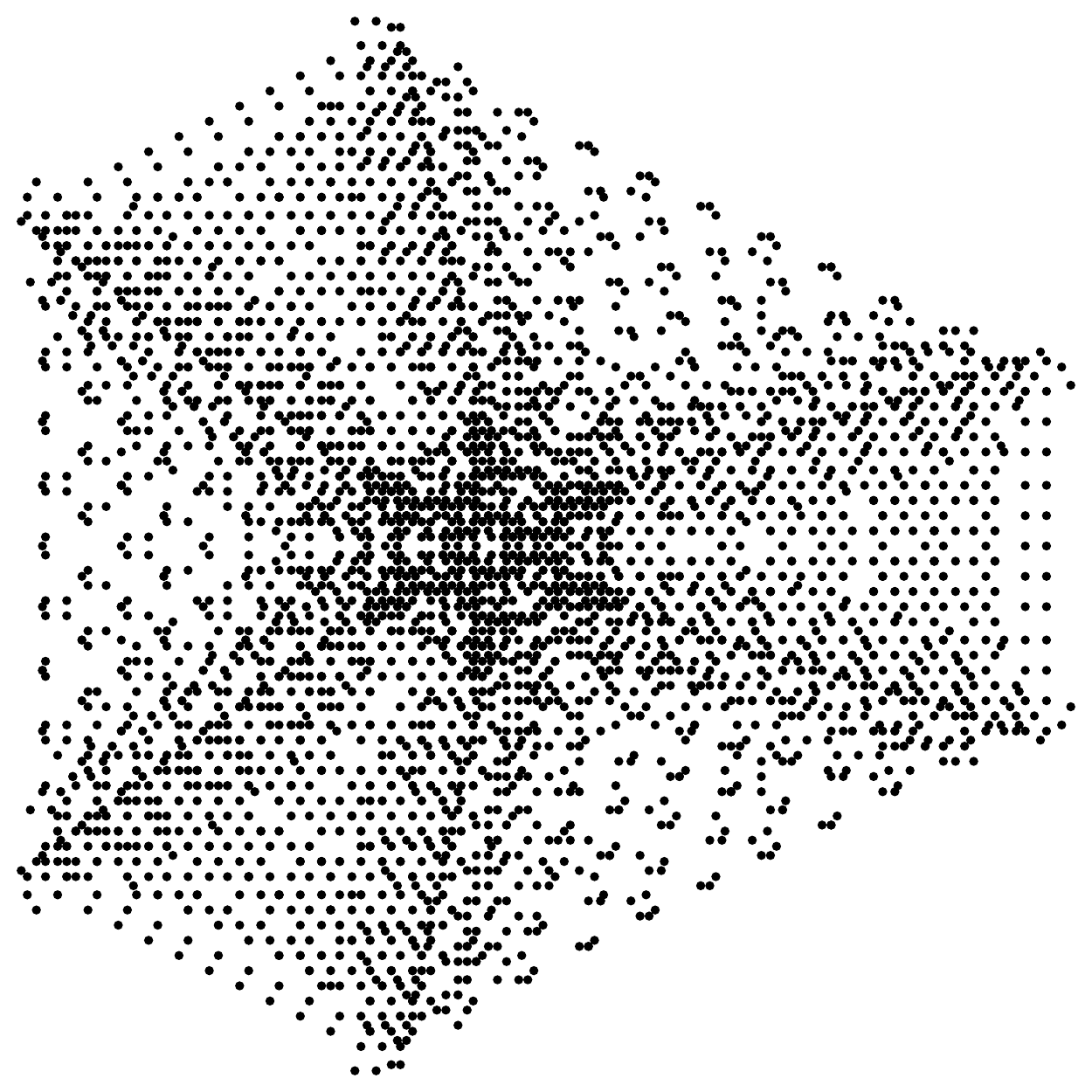}
\captionsetup{width = 0.95\textwidth}
\caption{Projection of the set onto the orthogonal complement of $(1,1,1)$.}
\label{fig:twodlattice}
\end{figure}
\end{center}
\vspace{-20pt}

\section{Independence of Norms and Universal Embeddings}

\subsection{Independence of Norms.} This section establishes a simple result that we believe to be fairly fundamental: Ulam sets depend strongly on initial
conditions but much less so on the notion of size that is used.  Note that the theorems throughout this section are not restricted to sets of initial elements in $\mathbb{R}^n$.  For a given set of initial elements $V=\{v_1, v_2, \dots v_k\}$, the the arising Ulam set $A$ is fully contained in the `domain' given by $$D_V=\{a_1v_1+a_2v_2+ \dots + a_k v_k: a_i \in \mathbb{N} \land \max(a_1, a_2, \dots, a_k) > 0\}$$ where some elements may have multiple expressions.  Our definition of Ulam sets uses the Pythagorean $\ell^2-$distance to determine the `smallest' element, but the arising sets are actually
independent of this particular notion of size.  More precisely, if the set of initial elements is $V=\{v_1, v_2, \dots, v_k\}$, then any function
$$f:D_V \rightarrow \mathbb{R}$$
satisfying
$$f(u+v)>\max{(f(u), f(v))} \qquad \text{for all} ~u,v\in D_V$$ such that $(-\infty, x)$ has a finite preimage for all $x \in \mathbb{R}$
will give rise to the same universal set independent of the particular function $f$ used. This property guarantees that it does not matter in what order we add new elements that are tied for smallest length. For example, all $\ell_p-$norms with $1 \leq p < \infty$ (including the Euclidean norm $\ell_2$) have this property,
as do all functions $f:\mathbb{R}_{\geq 0}^n \rightarrow \mathbb{R}$ that are strictly monotonically increasing and unbounded in each coordinate. For a given set of initial conditions $V=\left\{v_1, \dots, v_k\right\}$ and any admissible $f-$function, one can define the
$f-$Ulam set as the set obtained by adding, in a greedy manner, the smallest elements according to $f-$value that are uniquely representable as sums of two distinct earlier terms.
\begin{theorem}[Independence of Norms] $f-$Ulam sets are independent of the function $f$.
\end{theorem}
\begin{proof} The proof is by contradiction. Suppose that for a given set of initial elements $V=\{v_1, v_2, \dots v_k\}$, there are two functions $f_1, f_2$ giving rise to different Ulam sets, and suppose without loss of generality that $x \in D_V$ is a
$f_1-$smallest element that is contained in the first set but not the second set. Since the initial conditions and subsequently added elements are all contained in
$D_V$, whether the element $x$ is added to a set depends only on the elements
$$ \{a_1v_1+a_2v_2+\dots +a_kv_k\in D_V: x=b_1v_1+b_2v_2+\dots +b_kv_k ~\text{for some} ~b_i \in \mathbb{N}, b_i\geq a_i \forall 1\leq i \leq k\}.$$
(There could be many such expansions for $x$; we consider all elements $a_1v_1+a_2v_2+\dots +a_kv_k\in D_V$ satisfying this relation for at least one expansion of $x$.)  By assumption, this set coincides for $f_1$ and $f_2$, which yields a contradiction.
\end{proof}
This theorem shows the Ulam sets to depend on the underlying algebraic structure of addition but only very
weakly on the ordering by a notion of size. We remark that, in practice, the order in which elements are added to the set depends rather strongly on the $f-$function used, but the resulting (unordered) sets are ultimately the same.

\subsection{Universal Embeddings.} Suppose sets of initial conditions $U=\{u_1, u_2, \dots u_k\}$ and $V=\{v_1, v_2, \dots v_k\}$  give rise to Ulam sets $A \subset D_U$ and $B \subset D_V$, respectively.  Then we say that $A$ and $B$ are structurally equivalent (or isomorphic) as Ulam sets  if for all $a_i\in \mathbb{N}$,
$$a_1 u_1 + a_2 u_2 + \dots a_k u_k \in A \quad \text{iff} \quad a_1 v_1 + a_2 v_2 + \dots + a_k v_k \in B.$$

The next statement deals with a large number of initial cases and shows that the dynamics for `generic' initial conditions is unique. Here, `generic'
refers to the fact that for `most' (i.e. in the sense of Lebesgue measure) sets of initial vectors $\left\{v_1, \dots, v_k\right\} \subset  \mathbb{R}_{\geq 0}^n$, the equation
$$ a_1 v_1 + a_2 v_2 + \dots + a_k v_k =0 \qquad \mbox{has no solution}~(a_1, \dots, a_k) \in \mathbb{Z}^k \setminus 0.$$
In this case, things drastically simplify because each possible element of the set has a unique representation in $\{v_i\}$, which in turn implies universal behavior.

\begin{lemma}
 \label{indep}
 Let $V=\left\{v_1, \dots, v_k\right\}$ be a set of initial conditions such that
$$ a_1 v_1 + a_2 v_2 + \dots a_k v_k = 0 \qquad \mbox{has no solution}~(a_1, \dots, a_k) \in \mathbb{Z}^k \setminus 0.$$
Then the arising Ulam set $A$ is structurally equivalent to the Ulam set $E$ arising from the set
$ \left\{e_1, \dots, e_k\right\} \subset \mathbb{R}^k$
(where $e_i = (0,0,\dots, 0,1,0,\dots,0)$ is the $i-$th canonical basis vector in $\mathbb{R}^k$).
\end{lemma}
\begin{proof}
The argument is a straightforward application of Theorem 2: 
$$ a_1 v_1 + a_2 v_2 + \dots + a_k v_k =b_1 v_1 + \dots + b_k v_k$$
with $a_i, b_i \in \mathbb{N}$ implies that $a_i = b_i$ for all $1\leq i \leq k$ and thus that each potential element of the Ulam set has a unique representation in $\{v_i\}$ over $\mathbb{N}$ as coefficients.
We observe that we can, for all relevant lattice points that could ever be under consideration, define a $f-$function via
$$ f\left(a_1v_1+a_2v_2+\dots + a_kv_k\right) =\sqrt{a_1^2+a_2^2+\dots + a_k^2}.$$  
Theorem 2 implies that $A$ does not depend on the $f-$function, so the fact that this particular choice of $f-$function gives rise to the canonical Ulam set  $E$ concludes the argument.
\end{proof}
This simple statement has quite serious implications. For example, the Ulam set generated by
$\left\{(1,0), (1,\sqrt{2})\right\}$ behaves exactly like the set generated by $\left\{(1,0), (0,1)\right\}$ (which is
structurally fairly simple; see Figure 4). By the same token, the initial sets
$$\left\{(1,0,0), (1,\sqrt{2},0), (1,1,\sqrt{3})\right\} \subset \mathbb{R}_{\geq 0}^{3} \quad \text{and} \quad \{3, \sqrt{5}, 2+\pi \} \subset \mathbb{R}_{\geq 0}$$
both evolve exactly the same way as the initial set 
$$\left\{(1,0,0), (0,1,0), (0,0,1)\right\} \subset \mathbb{R}_{\geq 0}^{3}.$$

The following statement extends Lemma 2 and establishes even more general equivalence classes for Ulam sets (but is stated separately for clarity of exposition).
\begin{lemma}
 \label{indep} Suppose $U=\{u_1,\dots u_k\} $ and $V=\{v_1,\dots, v_k\}$ give rise to Ulam sets $A$ and $B$, respectively, under admissible notions of size $f_u$ and $f_v$.  Then $A$ and $B$ are structurally equivalent if $$a_1u_1+a_2u_2+\dots + a_ku_k=0$$
and
$$ a_1v_1+a_2v_2+\dots + a_kv_k=0$$
have the same set of solutions $(a_1, a_2, \dots a_k) \in \mathbb{Z}^k.$
\end{lemma}

\begin{proof}
We proceed by contradiction.  The above condition guarantees that representations of elements in $D_U$ and $D_V$ are completely equivalent: if an element of $D_U$ has multiple representations $a_1u_1+a_2u_1+\dots +a_ku_k=b_1u_1+b_2u_1+\dots +b_ku_k$ (with natural numbers as coefficients), then the corresponding element of $D_V$ has the same representations given by $a_1v_1+a_2v_1+\dots +a_kv_k=b_1v_1+b_2v_1+\dots +b_kv_k$.  Now suppose (without loss of generality) that there is a smallest element $x_0 \in A$ (measured according to $f_u$) representable by
$$a_1 u_1 + a_2 u_2 + \dots + a_k u_k = x_0 \in A$$
 such that the corresponding element 
$$ a_1 v_1 + a_2v_2 + \dots + a_k v_k = y_0 \notin B$$
and such that $A$ and $B$ agree for all elements that are strictly smaller than $x_0$ with respect to $f_u$. Since $x_0$ is included in $A$, it is the unique sum $x_0 = x_1 + x_2$ of two smaller elements $x_1, x_2 \in A$.  Thus, the corresponding elements $y_1$ and $y_2$ are included in $B$ and sum to $y_0$.  But in order for $y_0$ to be excluded from $B$, it must have a second representation: there exist $y_3, y_4 \in B$ such that $y_3+y_4=y_0$.  But by comparing coefficients, we see that the corresponding elements $x_3, x_4 \in A$ sum to $x_0$, a contradiction.
\end{proof}

In light of this lemma, it makes sense to call $$a_1u_1+a_2u_2+\dots + a_ku_k=0$$ the \textit{characteristic equation} for the initial conditions $\{u_i\}$.

\subsection{More General Objects.} These arguments easily extend to a more general setting. Suppose we have a set $\mathbb{A}$ of elements, a binary
operation $\circ:\mathbb{A} \times  \mathbb{A} \rightarrow \mathbb{A}$, a finite set of initial elements $\left\{a_1, a_2, \dots, a_k\right\}\subseteq \mathbb{A}$, and a function $f:\mathbb{A} \rightarrow \mathbb{R}$ that is acceptable (in the sense of \S3.1).
Then we can define an Ulam set arising from the initial set $\left\{a_1, a_2, \dots, a_k\right\}$ by repeatedly adding, among all elements with
a unique representation $a = a_i \circ a_j$ ($i \neq j$), one with minimal value of $f$.  (If $\circ$ is commutative, then the canonical definition restricts our consideration to sums of unique pairs of elements.) As before, the order in which we break ties is inconsequential because of the constraint on $f$, and the same argument as above implies that the arising set is independent of the function $f$.  \\

\textit{Example 1.} Let $\left\{A_1, \dots, A_k\right\} \subset \mathbb{R}^{n \times n}$ be a set of pairwise commuting $n \times n$ matrices with $\det(A_i) > 1$ for all $1 \leq i \leq k$.  Then define our binary operation as standard matrix multiplication,  and let
$$ f(A) := \det(A)$$
be our notion of size.  
It is easy to see that $$f(AB) = \det(AB) = \det(A)\det(B) > \max(\det(A), \det(B)) = \max(f(A), f(B)).$$ 
Since all of the matrices commute, there exists a change of basis under which they all become upper triangular, and this
property is preserved under multiplication. \\

\textit{Example 2.} Let $\left\{g_1, \dots, g_k\right\} \subset C([0,1], \mathbb{R}_{> 0})$ be a set of continuous functions each enclosing strictly positive area, let the binary operation $\circ$ be given by addition, and set
$$ f(g) := \int_{0}^{1}{g(x) dx}.$$
For instance, the Ulam set arising from $\left\{1, \sin{x}, \cos{x}\right\}$ is isomorphic to the set obtained from $\left\{(1,0,0), (0,1,0), (0,0,1)\right\} \subset \mathbb{R}^3$, and the set arising from $\{ \sin^2x, \cos^2x, 1\}$ is isomorphic to the set arising from $\{ (1,0), (0,1), (1,1)\} \subset \mathbb{R}^2$.

\subsection{Embedding Into the Real Line.} The previous section shows that many general initial conditions can be reduced to universal
dynamical behavior on a lattice, the dimension of which is determined by the initial values.  Perhaps surprisingly, one can also reduce dynamical behavior to the case of one-dimensional Ulam sequences (possibly with real initial conditions). We believe this to be one of the reasons why Ulam sequences with non-integer initial elements or more than two initial elements have never been actively investigated: the underlying dynamics can be
of a higher-dimensional nature.

\begin{lemma} Let $\left\{v_1, \dots, v_k\right\} \subset \mathbb{Z}_{\geq 0}^n$ be a set of nonzero vectors. The arising Ulam set $A$ is isomorphic to a suitable one-dimensional Ulam set.
\end{lemma}
\begin{proof} The proof is constructive. We map each vector to a unique positive real number via
$$ \phi(x) = \phi(x_{1}, x_2, \dots, x_n) := \log{\left( 2^{x_1} 3^{x_2} \dots p_n^{x_n} \right)},$$
where $p_i$ is the $i-$th prime number. We now claim that the set
$$ \left\{  \phi(v_1), \dots, \phi(v_k) \right\} \subset \mathbb{R}_{\geq 0},$$
interpreted as the initial conditions of a one-dimensional Ulam sequence, exhibits the same dynamics. It is clear that
$$ \phi(u) + \phi(v) = \phi(u+v)$$
is equivalent to the additive relationship of the vectors in $\mathbb{N}^n$. It remains only to note that
$\phi$ can be interpreted as a continuous function $\phi:\mathbb{R}^n_{\geq 0} \rightarrow \mathbb{R}$ that is strictly
monotonically increasing and unbounded in each of its coordinates.
\end{proof}

\subsection{Embedding Into the Integer Lattice}
We also prove a converse to the above Lemma 4, namely, that an acceptable set of initial conditions in $\mathbb{R}^m$ is always structurally equivalent to some initial conditions in $\mathbb{Z}_{\geq 0}^l$.  This embedding is useful in two ways: first, it greatly narrows the search space for Ulam sets; and second, it allows us to apply the results of \S 4 which we derive for initial conditions with integer coordinates.

\begin{lemma}
Let $V=\{v_1, v_2, \dots v_k \}\subset \mathbb{R}^m_{\geq 0}$ be  a set of initial conditions where each $v_i$ is nonzero and has all nonnegative components.  Then there exists some set of initial conditions \linebreak $W=\{w_1, w_2, \dots w_k\}\subset \mathbb{N}^l$ with $1 \leq l \leq k$ that gives rise to a structurally equivalent Ulam set.
\end{lemma}
\begin{proof}
To begin, define the rational solutions to the characteristic equation of $V$ to be $$\mathcal{S}=\{(a_1, a_2, \dots a_k)\in \mathbb{Q}^k: a_1v_1+a_2v_2+\dots +a_kv_k=0\}.$$  As shown in Lemma 3, $\mathcal{S}$ completely determines the behavior of the Ulam set arising from $V$.  (Because this equation is homogeneous, we may take integer and rational solutions interchangeably.)  Now, let $Q=\{q_1, q_2, \dots q_l\}$ be a minimal subset of $V$ such that every element of $V$ is expressible as a $\mathbb{Q}-$linear combination of the elements of $Q$.  ($1 \leq l \leq k$ is clear from any construction of $Q$.  The cases $|Q|=1$ and $|Q|=k$ are special: the former is equivalent to a one-dimensional Ulam sequence with integer coeffients, and the latter exhibits the universal behavior of Lemma 2.)  By the minimality of $Q$, each $v_i \in V$ can be uniquely written as $$v_i=c_1^iq_1+v_2^iq_2+\dots +c_l^iq_l$$ where the coefficients are rational.  Then the characteristic equation becomes $$a_1(c^1_1 q_1+c^1_2 q_2+\dots +c^1_l q_l)+a_2(c^2_1 q_1+c^2_2 q_2+\dots +c^2_l q_l)+\dots +a_k(c^k_1 q_1+c^k_2 q_2+\dots +c^k_l q_l)=0.$$  Again, the minimality of $Q$ ensures that the $q_i$'s are $\mathbb{Q}-$linearly independent, so we may separate this equation into the system of $l$ simultaneous equations in $k$ variables given by :
$$\begin{bmatrix}
c^1_1&c^2_1& \dots &c^k_1 \\
c^1_2&c^2_2& \dots &c^k_2 \\
\vdots & \vdots & \ddots & \vdots \\
c^1_l&c^2_l& \dots &c^k_l
\end{bmatrix}
\begin{bmatrix}
a_1\\
a_2\\
\vdots \\
a_k
\end{bmatrix}=
\begin{bmatrix}
0\\
0\\
\vdots \\
0
\end{bmatrix}$$
Interpreting the columns of this matrix as the expansions of vectors in $\mathbb{R}^l$ over the standard basis, we see that the set of initial conditions $U=\{u_i=(c_1^i, c_2^i, \dots c_l^i)\}\subset \mathbb{Q}^l$ (for $1 \leq i \leq k$) gives the same rational solutions to the characteristic equation as $V$.   However, the proof is not complete because some elements of $U$ may have negative components.  For each $1 \leq i \leq l$, define $b_i$ to be the sum of the components of $q_i$.  Since each $q_i$ is nonzero and has all nonegative components, each $b_i$ is strictly positive.  We now associate any vector $(d_1, d_2, \dots d_l)\in \mathbb{R}^l$ with a real number via the linear map $$\psi (d_1, d_2, \dots d_l)=d_1b_1+d_2b_2+\dots +d_lb_l.$$  Note that $\psi (u_i)$ equals the sum of the coordinates of $v_i$ and hence is strictly positive.  Thus, the solutions to $$\psi(x_1, x_2, \dots x_l)=x_1b_1+x_2b_2+\dots +x_lb_l=0$$ form a $(l-1)-$dimensional hyperplane bisecting $\mathbb{R}^l$, and all the $u_i$'s lie strictly on the same side of it.  We remark that the vector $u_{\bot}=(b_1, b_2, \dots b_l) \in \mathbb{R}^l$ is orthogonal to the hyperplane and  lies on the `positive side,' i.e., the angle between $u_{\bot}$ and each $u_i$ is strictly smaller than $\frac{\pi}{2}$.  Because the inequality is strict, there is some `wiggle room' around $u_{\bot}$ such that all the angles remain  strictly smaller than $\frac{\pi}{2}$.  More precisely, there exists an open ball around the endpoint of $u_{\bot}$ such that any vector with its endpoint in that ball retains said property.  By the denseness of the rationals, we can find such a $u'_{\bot}$ with all rational components that is `almost' perpendicular to the boundary space.  It follows that the dot product of $u'_{\bot}$ with each $u_i$ is always strictly positive.  Let $$m=\min_{1 \leq i \leq k} u_i \cdot u'_{\bot} \quad \text{and} \quad \Omega=\min_{\substack{1 \leq i \leq k \\ 1 \leq j \leq l}} c^i_j$$ where $m>0$ is guaranteed but $\Omega$ may be negative.  Then there exists a positive integer $M$ such that $\Omega +Mm>0$.  We now create a new set $Y=\{y_i\}\subset \mathbb{R}^l$ defined by
$$y_i=u_i+M(u_i \cdot u'_{\bot})(1,1,\dots 1)=(I_l+M\textbf{1}(u'_{\bot})^{T})u_i$$
where $I_l$ is the $(l \times l)$ identity matrix and $\textbf{1}=(1,1,\dots 1)$ is a column vector.  By construction, each $y_i$ has all components strictly positive.  Moreover, since $u_i \cdot u'_{\bot}$ is rational, each $\{y_i\}$ has rational components.  Scaling the entire set by the least common denominator gives a set \linebreak $W=\{w_i\} \subset \mathbb{Z}_{\geq 0}^l$.  We now need note only that the transformation from $u_i$ to $w_i$ is an invertible linear transformation, which means that it preserves $\mathcal{S}$.  (The transformation matrix is singular only if $-\frac{1}{M}$ is an eigenvalue of the matrix $\textbf{1}(u'_{\bot})^{T}$.  But this matrix has only finitely many eigenvalues, so if the smallest integer $M$ makes the transformation noninvertible, then a larger value of $M$ can be chosen.)  This completes the proof.
\end{proof}

\section{The Column Phenomenon}
This section is devoted to a curious phenomenon that we first observed in the set arising from $\{(1,0), (2,0), (0,1) \}$ (the classical Ulam sequence on the $x-$axis augmented by a vector in the orthogonal direction).
The picture (Figure \ref{fig:ulammed}) is rather stunning: seemingly chaotic behavior close to the $x-$axis and periodic structures evolving in the direction of the $y-$axis. The list of $x-$coordinates for which nonempty columns arise is given by
$$1, 4, 6, 9, 14, 20, 23, 25, 30, 33, 49, 56, 60, 248,  270, 280, 302, 385, 474, 479, \dots$$
At this point, we do not understand whether and how this sequence evolves further.
This example naturally leads to defining a column as, loosely, a structure that periodically extends to infinity in one direction.  The purpose of this section is to prove the existence of such periodic structures.  We begin by considering columns in 2 dimensions extending in the direction of the $y-$ axis. After proving several results about the behavior of these columns, we provide natural generalizations to more complex column behavior.

\begin{center}
\begin{figure}[h!]
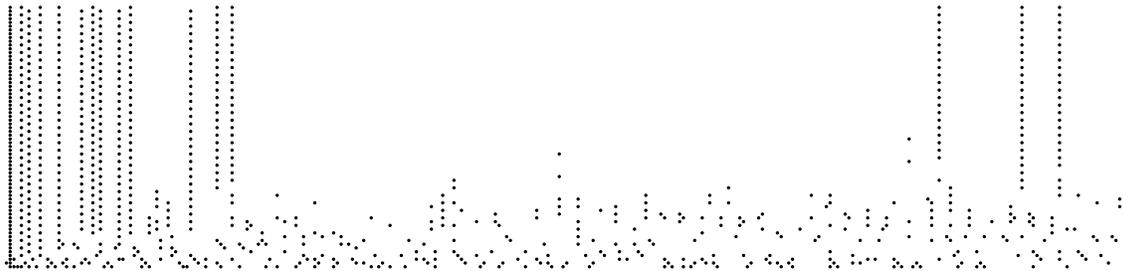


\captionsetup{width = 0.95\textwidth}
\caption{The set arising from $\left\{(1,0), (2,0), (0,1)\right\}$.  We see a gap in nonempty columns between $x=60$ and $x=248$.  Two more nonempty columns follow at $x=270$ and $x=280$.}
\label{fig:ulammed}
\end{figure}
\end{center}
\vspace{-20pt}

\subsection{A Combinatorial Lemma}
Our proof of the existence of columns requires a simple combinatorial fact that we prefer to state independently.  Let $X$ be the set of infinite words over the alphabet $\left\{0,1,2\right\}$ and $Y$ the set of infinite words over $\left\{0,1\right\}$.  We define a transformation $T: X \rightarrow Y$ by setting $$T(x)_0:=
\begin{cases}
1 &\text{if}~ x_0=1\\
0 &\text{otherwise}.
\end{cases}$$
and, for $i \geq 1$,
$$T(x)_i:=
\begin{cases}
1 &\text{if}~ x_i+T(x)_{i-1}=1\\
0 &\text{otherwise}.
\end{cases}$$
We are interested in how this map affects infinite words that are eventually periodic.  For example,
$$110011001\dots \xrightarrow[]{T}  100010001\dots$$ 
is a word with period 4 that is mapped to another word with period 4. By contrast,
$$010101010\dots  \xrightarrow[]{T}   011001100\dots$$ 
 is a word with period 2 mapped to a word with period 4. Our next statement shows this
to be an exhaustive case distinction.

\begin{lemma}[$T$ Preserves Periodicity.] If $x \in X$ is eventually periodic with period $p$ (i.e. $x_{j+p} = x_j$ for all $j$ sufficiently large), then $T(x)$ is eventually periodic with period either $p$ or $2p$. $T(x)$ is periodic with period $2p$ iff the each period in the periodic portion of $x$ contains an odd number of 1's and no 2's. \label{lemma:comb}
\end{lemma}
\begin{proof} Fix $m$ sufficiently large for the word $x$ to be periodic after the first $m$ symbols, and consider the action of $T$ on the three quantities $x_m, x_{m+p}, x_{m +2p}$ (all three of which
are identical because of the eventual periodicity of $x$). If $T(x)_m = T(x)_{m+p}$, then we clearly have $T(x)_{m+1} = T(x)_{m+p+1}$ since $T(x)_{i}$ depends on only $x_i$ and $T(x)_{i-1}$. It follows by induction that $T(x)$ becomes periodic with period $p$.
Suppose now that $T(x)_m \neq T(x)_{m+p}$. Then, since there are only two symbols in $Y$, we have either $T(x)_{m+2p} = T(x)_m$ or $T(x)_{m+2p} = T(x)_{m+p}$. The second case is identical to the previously considered case and implies that $T(x)$ will be periodic with period $p$. If $T(x)_m = T(x)_{m + 2p},$ then we can infer that $T(x)$ is periodic with period $2p$.\\
Now we come to the second part of the statement. If $x_j = 2$ for some $j$ in the periodic section of $x$, then we must have $T(x)_j = 0$ (and therefore the $p-$periodicity of $x$ implies $T(x)_j = 0 = T(x)_{j+p}$), so we obtain that $T(x)$ is periodic with period $p$. We may thus limit ourselves to considering periodic words in $x$ containing only 0's and 1's in the periodic section. We observe for $l \geq m$ that flipping the value of $x_l$ (either from 0 to 1 or from 1 to 0) has the effect of flipping all values $T(x)_i$ for $i \geq l$. We can thus start with the basic word
$$0000\dots  \xrightarrow[]{T}  0000\dots$$
and add 1's one-by-one until we re-create the original string $x$: we place 1's in the proper position starting from the top of the range we are interested in and work our way down. If there is an even number of 1's in each period of $x$, then $T(x)_{m+p} = T(x)_{m}$, whereas an odd number of 1's in each period implies $T(x)_{m + 2p} = T(x)_m \neq T(x)_{m+p}$. 
\end{proof}

\textit{Remark.} We have shown only that $T(x)$ must be periodic with period either $p$ or $2p$.  The minimal period, however, can be any divisor of $p$ (in the first case) or any divisor of $2p$ but not of $p$ (in the second case).

\subsection{Existence and Doubling}

This section is devoted to an analysis of columns in the two-dimensional case. We consider an initial set $\left\{v_1, v_2, \dots, v_k\right\} \subset \mathbb{Z}_{\geq 0}^2$ with the property
that one of the vectors is $(0,1)$ and no other vector lies on the $y-$axis. In this setting, the columns naturally extend in the direction of the $y-$axis. We will say that the set has a column with period $p$ over $x \in \mathbb{N}$ if, for $y$ large enough,
$$(x,y)~\text{is in the Ulam set iff}~(x,y+p)~\mbox{is also in the set}.$$
We will usually talk about the period of a column and not necessarily the minimal period. Moreover, for fixed $x \in \mathbb{N}$, the case where there are no $(x,y)$ in the set for
$y$ beyond a certain threshold will also be denoted a column (the empty column).  We now show that columns extend all the way to infinity: for every fixed $x \in \mathbb{N}$, the behavior along the $y-$axis ultimately becomes periodic with a period that is a power of 2.

\begin{center}
\begin{figure}[h!]
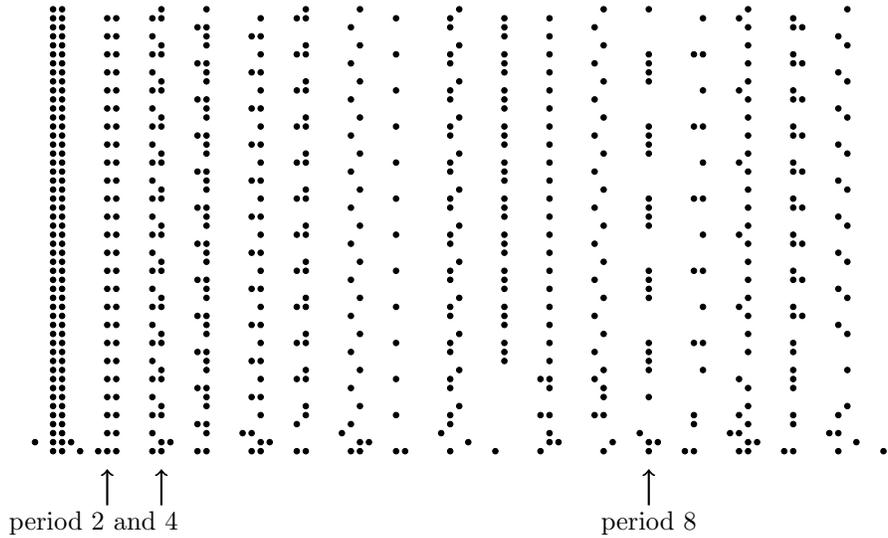


\captionsetup{width = 0.95\textwidth}
\caption{The set arising from $\left\{(2,0), (3,0), (0,1)\right\}$. This example shows columns of period 1, 2, 4 and 8, and we mark the first occurrences of periods 2,4, and 8.}
\label{fig:doubling}
\end{figure}
\end{center}
\vspace{-15pt}

\begin{theorem}[Periodicity in the $y-$direction] Let $\left\{v_1, v_2, \dots, v_k\right\} \subset \mathbb{Z}_{\geq 0}^2$ contain $(0,1)$ and no other vector on the $y-$axis. Then there exists a function $\phi:\mathbb{N} \rightarrow \mathbb{N}$ such that a nonempty column extends over $x$ if and only if there is an element $(x,y)$ in the set with $y \geq \phi(x)$. All columns (including empty columns) are eventually periodic and the period is a power of $2$. Moreover, the period is either the period of a preceding column or twice the period of a preceding column.
\end{theorem}

\begin{proof}
The proof proceeds by induction. 
The set clearly contains no points $(0,n)$ with $n \geq 2$, which means that an empty column extends over $x=0$ with period $2^0$.  We now assume that the statement
is true up to some $x-1$ and investigate the possible behavior of the set for lattice points with first coordinate fixed to be $x$.
We consider vectors of the form $(x,y)$ for $y$ much larger than any of the previously obtained bounds $\phi(0), \phi(1), \phi(2), \dots, \phi(x-1)$ and the $y-$values of any possible initial vectors with $x-$coordinate $x$.  We want to use these elements to show the existence of an infinitely periodic column over $x$. To that end, we first completely ignore the existence of the vector $(0,1)$ in the set and obtain a complete description without it;
we then add $(0,1)$ and explain its effect using Lemma \ref{lemma:comb}. We begin by arguing that any vector that is the sum of two elements from preceding columns with
$y-$coordinates significantly larger than the cutoff function $\phi$ has at least 2 representations. In this case, the periodicity of the preceding columns means that once there is a single representation, a second representation of the point can also be found easily.  (By taking points significantly larger than the previous $\phi -$bounds, we can work in the regime where at least one summand comes from well within the periodic region.)  At the same time, $2 \max_{0 \leq i \leq x-1} \phi (i)$ bounds the $y-$coordinate of any sum of two elements each with second coordinate smaller than $\phi$, then we can exclude this case by moving past that number.

\begin{center}
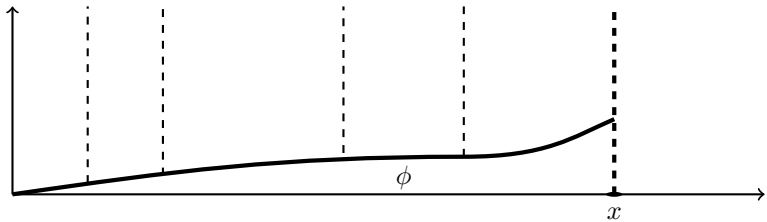
\begin{figure}[h!]
\begin{tikzpicture}[xscale=2,yscale=0.5]
\draw [thick,->] (0,0) -- (5,0);
\draw [thick,->] (0,0) -- (0,5);
\filldraw (4,0) circle (0.05cm);
\node at (4,-0.5) {$x$};
\node at (2.6,0.45) {$\phi$};
\draw [thick, dashed] (0.5,0.3) -- (0.5,5);
\draw [thick, dashed] (1,0.5) -- (1,5);
\draw [thick, dashed] (2.2,1.1) -- (2.2,5);
\draw [thick, dashed] (3,1) -- (3,5);
\draw[ultra thick] (0,0) to[out=30, in =180] (3,1)  to[out=0, in =240] (4,2);
\draw [dashed, ultra thick] (4,0) -- (4,5);
\end{tikzpicture}
\captionsetup{width = 0.95\textwidth}
\caption{A splitting into two regions.}
\end{figure}
\end{center}
\vspace{-20pt}
This implies that for $y$ sufficiently large, any hypothetical element $(x,y)$ can, if it exists, be uniquely written as 
$$ (x,y) = (x_1, y_1) + (x_2, y_2)$$
with $y_1 \geq \phi(x_1)$ and $y_2< \phi(x_2)$. This uniqueness, along with the periodicity of the column over $x_1$, then implies (denoting the period of the column over $x_1$ by $p$) that
$$ (x,y+p) = (x_1, y_1 + p) + (x_2, y_2) \qquad \mbox{uniquely.}$$
Likewise, the existence of a second representation for $(x, y+p)$ would automatically create a second representation for $(x,y)$, which is a contradiction. This implies that the `pre-correction-column' over $x$ has the
same period as that of a preceding column (which, by induction, is a power of 2). An application of Lemma \ref{lemma:comb} then accounts for the additional vector $(0,1)$ and shows that the
period may double, which would yield another power of 2.
\end{proof}


\textit{Remarks.} 
\begin{enumerate}

\item We note that in the case where doubling does not occur, the minimal period of the post-correction column may be a smaller power of 2 than the minimal period of the pre-correction column (due to Lemma \ref{lemma:comb}), but it still matches the minimal period of some previous column.

\item Another immediate consequence of this application of Lemma \ref{lemma:comb} is that the periodic portion of any column that doubles from period $2^n$ to period $2^{n+1}$ has the property that exactly one of the points $(x,y)$ and $(x, y+2^n)$ is included in the Ulam set (for sufficiently large $y$).

\item A careful inspection of the proof of Theorem 3 allows us to derive that $\phi(n) \leq c \cdot 3^n$ for some constant $c$ depending on the initial vectors. However, in practice $\phi$ seems to be much, much smaller, and we consider it an interesting problem to gain a better understanding of in which regions periodicity starts being enforced. (Numerically, it does seem that $\phi$ could very well be linear or
at most polynomial in most cases.)
\end{enumerate}

\subsection{Generalizations.} One notes that the above discussion applies equally well to columns extending in the direction of the $x-$axis arising due to the action of the initial vector $(1,0)$.  This structure result thus applies to sets of the form $\left\{(1,0), (0,1), v_3, \dots, v_k\right\} \subset \mathbb{Z}_{\geq 0}^2$
where all the vectors $v_3, \dots, v_k$ have both coordinates strictly positive.  Thus, there may be regions where all elements of the Ulam set are periodic in both the $x-$ and $y-$directions.  (See Figure \ref{fig:gencolumn}.)     
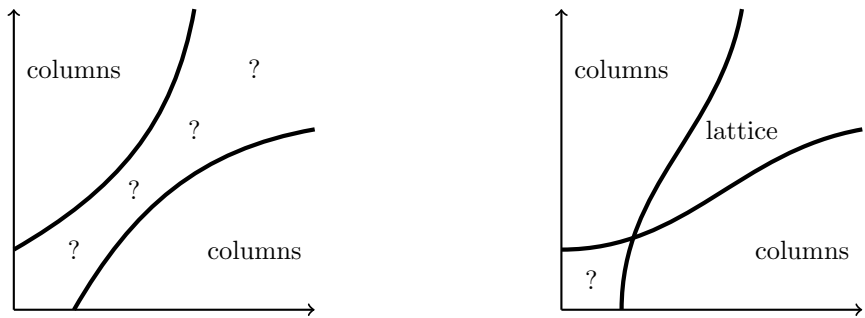
\begin{figure}[h!]
\begin{minipage}[l]{0.48\linewidth}
\centering
\begin{tikzpicture}[scale=0.8]
\draw [thick, ->] (0,0) -- (5,0);
\draw [thick, ->] (0,0) -- (0,5);
\draw[ultra thick] (0,1) to[out=30, in =260] (3,5);
\draw[ultra thick] (1,0) to[out=60, in =190] (5,3);
\node at (1,4) {columns};
\node at (4,1) {columns};
\node at (1,1) {?};
\node at (2,2) {?};
\node at (3,3) {?};
\node at (4,4) {?};
\end{tikzpicture}
\end{minipage}
\begin{minipage}[r]{0.48\linewidth}
\centering
\begin{tikzpicture}[scale=0.8]
\draw [thick, ->] (0,0) -- (5,0);
\draw [thick, ->] (0,0) -- (0,5);
\draw[ultra thick] (0,1) to[out=0, in =190] (5,3);
\draw[ultra thick] (1,0) to[out=90, in =260] (3,5);
\node at (1,4) {columns};
\node at (4,1) {columns};
\node at (0.5,0.5) {?};
\node at (3,3) {lattice};
\end{tikzpicture}
\end{minipage}
\captionsetup{width = 0.95\textwidth}
\caption{Regions bounded by $y=\phi(x)$ and $x=\phi^{}(y)$ and two possible types of behavior.  It is also conceivable for the curves to intersect multiple times (not shown here).}
\label{fig:gencolumn}
\end{figure}
\linebreak If $\{v_1, v_2, \dots, v_k\} \subset \mathbb{Z}_{\geq 0}^2$ contains $(0,a)$ (where $a$ is any positive integer) but no other vector on the $y-$axis, then the argument above still applies with the main difference being that 
columns are now periodic with periods $a\cdot 2^n$: if we split the $y-$coordinates with respect to their residue class modulo $a$, the proof above essentially applies verbatim.
 We also remark that the argument extends easily to extremal directions in high-dimensional cases, and we leave the details to the interested reader.

\section{Open Problems}

It is clear that Ulam sets are incredibly rich in structure and that we have barely managed to scrape the surface. Among the many natural questions, we explicitly point
out a few that seem particularly promising for future investigation.

\subsection{Higher-dimensional Examples.} We have been almost exclusively concerned with examples in $\mathbb{R}^2_{\geq 0}$ and the example
$$ \left\{(1,0,0), (0,1,0), (0,0,1) \right\} \subset \mathbb{R}^3_{\geq 0}.$$
We emphasize that this three-dimensional example provides the universal dynamics for `generic' sets of three elements in fairly general sets $\mathbb{A}$ (including
the generic case of Ulam-type sequences in $\mathbb{R}_{\geq 0}$ with an initial set of three elements that are linearly independent over $\mathbb{Z}$). This set seems to have an extraordinary amount
of structure and symmetry and should be of great interest.

\subsection{Lattice Structures.} One natural question is how unavoidable lattice structures are: if the initial set $\left\{v_1, \dots, v_k\right\} \subset \mathbb{Z}_{\geq 0}^2$ has
exactly one vector of the form $(x,0)$ and exactly one vector of the form $(0,y)$, must the resulting set ultimately exhibit lattice-type structure for sufficiently large $x$ and $y$? 
Conversely, does the existence of aperiodic behavior (for instance, the classical Ulam sequence) on one axis always preclude regular column structures in that direction?

\subsection{Classical Ulam Sequence and $(0,1)$.}  Clearly, one of the most striking examples is given by $\left\{(1,0), (2,0), (0,1)\right\}$. As already discussed, we recover the classical Ulam sequence on the $x-$axis, then we see fairly intricate behavior close to the $x-$axis and occasional nonempty vertical columns. We observe that in the first 100.000 elements of this sequence, no column seems to have period larger than 2. We also observe that all elements
in the sequence with $x-$coordinate fixed ($x \geq 2$) have their second coordinate either always even or always odd. If true, this would imply that all nonempty columns are of period 2. (See \S 4.2.)  We also observe that the
Fourier frequency phenomenon from \cite{stein} seems to appear if we look at points for which the $y-$coordinate is fixed. It would be quite fascinating if the dynamical behavior of this Ulam set could shed some light on the classical Ulam sequence in one dimension.

\subsection{Columns and Their Properties} Figure \ref{fig:doubling} shows the first few points in the evolution of $\left\{(2,0), (3,0), (0,1)\right\}$. We observe doubling of
the column period three times in the elements we have calculated so far. While we have shown that columns (empty or nonempty) eventually arise, we do not understand
columns very well: when are there infinitely many nonempty columns? Do they generically double their period (as observed in Figure \ref{fig:doubling}), or do they not (which seems to be
the case for the extended Ulam sequence)?

\subsection{Classification.} A complete classification of the behavior of $\left\{(1,0), (0,1), (m,n)\right\}$ for the case $m,n \in \mathbb{Z}_{\geq 1}$ seems within reach (and, using linear invariance,
naturally includes many other initial conditions as well).  (Note added in revision: this question has in the meanwhile been completely resolved by Hinman, Kuca, Schlesinger \& Sheydvasser \cite{hinman}.)  Here, we sketch the general types of structures that arise (excluding some `edge cases' where $(m,n)$ is small). If $(m,n)$ is included in the lattice generated by $(1,0)$ and $(0,1)$, then the Ulam sequence is degenerate: the new vector does not 
contribute anything and we obtain the lattice generated by $\{(1,0), (0,1) \}$ alone.  
\begin{enumerate}
\item If both coordinates of $(m,n)$ are even, we observe that the resulting set consists of a series of repeating, equally spaced $L-$shaped figures parallel to the coordinate axes.  More specifically, all elements of the $L$'s have both coordinates odd, and each $L$ consists of $m/2$ horizontal columns of period 2 and $n/2$ vertical columns of period 2, with all columns spaced 2 apart.  Interestingly, this means that many of the columns exhibit periodic behavior for the intervals before they begin exhibiting the pattern that continues to infinity. It also shows that the function $\phi$ (such that $y \geq \phi(x)$ forces periodicity) can
grow linearly.

\begin{center}
\begin{figure}[h!]
\begin{tikzpicture}[scale=0.11]

\filldraw (1,0) circle (0.3cm);
\filldraw (0,1) circle (0.3cm);
\filldraw (6,4) circle (0.3cm);
\filldraw (1,1) circle (0.3cm);
\filldraw (2,1) circle (0.3cm);
\filldraw (1,2) circle (0.3cm);
\filldraw (3,1) circle (0.3cm);
\filldraw (1,3) circle (0.3cm);
\filldraw (4,1) circle (0.3cm);
\filldraw (1,4) circle (0.3cm);
\filldraw (3,3) circle (0.3cm);
\filldraw (5,1) circle (0.3cm);
\filldraw (1,5) circle (0.3cm);
\filldraw (3,5) circle (0.3cm);
\filldraw (5,3) circle (0.3cm);
\filldraw (6,1) circle (0.3cm);
\filldraw (1,6) circle (0.3cm);
\filldraw (7,1) circle (0.3cm);
\filldraw (1,7) circle (0.3cm);
\filldraw (5,5) circle (0.3cm);
\filldraw (3,7) circle (0.3cm);
\filldraw (7,3) circle (0.3cm);
\filldraw (8,1) circle (0.3cm);
\filldraw (1,8) circle (0.3cm);
\filldraw (5,7) circle (0.3cm);
\filldraw (9,1) circle (0.3cm);
\filldraw (1,9) circle (0.3cm);
\filldraw (3,9) circle (0.3cm);
\filldraw (9,3) circle (0.3cm);
\filldraw (10,1) circle (0.3cm);
\filldraw (1,10) circle (0.3cm);
\filldraw (5,9) circle (0.3cm);
\filldraw (11,1) circle (0.3cm);
\filldraw (1,11) circle (0.3cm);
\filldraw (3,11) circle (0.3cm);
\filldraw (11,3) circle (0.3cm);
\filldraw (12,1) circle (0.3cm);
\filldraw (1,12) circle (0.3cm);
\filldraw (5,11) circle (0.3cm);
\filldraw (13,1) circle (0.3cm);
\filldraw (1,13) circle (0.3cm);
\filldraw (3,13) circle (0.3cm);
\filldraw (13,3) circle (0.3cm);
\filldraw (5,13) circle (0.3cm);
\filldraw (14,1) circle (0.3cm);
\filldraw (1,14) circle (0.3cm);
\filldraw (15,1) circle (0.3cm);
\filldraw (1,15) circle (0.3cm);
\filldraw (3,15) circle (0.3cm);
\filldraw (15,3) circle (0.3cm);
\filldraw (5,15) circle (0.3cm);
\filldraw (13,9) circle (0.3cm);
\filldraw (16,1) circle (0.3cm);
\filldraw (1,16) circle (0.3cm);
\filldraw (17,1) circle (0.3cm);
\filldraw (1,17) circle (0.3cm);
\filldraw (13,11) circle (0.3cm);
\filldraw (3,17) circle (0.3cm);
\filldraw (17,3) circle (0.3cm);
\filldraw (15,9) circle (0.3cm);
\filldraw (5,17) circle (0.3cm);
\filldraw (18,1) circle (0.3cm);
\filldraw (1,18) circle (0.3cm);
\filldraw (13,13) circle (0.3cm);
\filldraw (15,11) circle (0.3cm);
\filldraw (19,1) circle (0.3cm);
\filldraw (1,19) circle (0.3cm);
\filldraw (3,19) circle (0.3cm);
\filldraw (19,3) circle (0.3cm);
\filldraw (17,9) circle (0.3cm);
\filldraw (5,19) circle (0.3cm);
\filldraw (13,15) circle (0.3cm);
\filldraw (15,13) circle (0.3cm);
\filldraw (20,1) circle (0.3cm);
\filldraw (1,20) circle (0.3cm);
\filldraw (17,11) circle (0.3cm);
\filldraw (21,1) circle (0.3cm);
\filldraw (1,21) circle (0.3cm);
\filldraw (19,9) circle (0.3cm);
\filldraw (3,21) circle (0.3cm);
\filldraw (21,3) circle (0.3cm);
\filldraw (15,15) circle (0.3cm);
\filldraw (13,17) circle (0.3cm);
\filldraw (17,13) circle (0.3cm);
\filldraw (5,21) circle (0.3cm);
\filldraw (19,11) circle (0.3cm);
\filldraw (22,1) circle (0.3cm);
\filldraw (1,22) circle (0.3cm);
\filldraw (15,17) circle (0.3cm);
\filldraw (17,15) circle (0.3cm);
\filldraw (21,9) circle (0.3cm);
\filldraw (23,1) circle (0.3cm);
\filldraw (1,23) circle (0.3cm);
\filldraw (13,19) circle (0.3cm);
\filldraw (3,23) circle (0.3cm);
\filldraw (23,3) circle (0.3cm);
\filldraw (5,23) circle (0.3cm);
\filldraw (21,11) circle (0.3cm);
\filldraw (24,1) circle (0.3cm);
\filldraw (1,24) circle (0.3cm);
\filldraw (17,17) circle (0.3cm);
\filldraw (15,19) circle (0.3cm);
\filldraw (23,9) circle (0.3cm);
\filldraw (13,21) circle (0.3cm);
\filldraw (25,1) circle (0.3cm);
\filldraw (1,25) circle (0.3cm);
\filldraw (3,25) circle (0.3cm);
\filldraw (25,3) circle (0.3cm);
\filldraw (5,25) circle (0.3cm);
\filldraw (23,11) circle (0.3cm);
\filldraw (17,19) circle (0.3cm);
\filldraw (15,21) circle (0.3cm);
\filldraw (26,1) circle (0.3cm);
\filldraw (1,26) circle (0.3cm);
\filldraw (13,23) circle (0.3cm);
\filldraw (25,9) circle (0.3cm);
\filldraw (27,1) circle (0.3cm);
\filldraw (1,27) circle (0.3cm);
\filldraw (17,21) circle (0.3cm);
\filldraw (3,27) circle (0.3cm);
\filldraw (27,3) circle (0.3cm);
\filldraw (25,11) circle (0.3cm);
\filldraw (5,27) circle (0.3cm);
\filldraw (15,23) circle (0.3cm);
\filldraw (28,1) circle (0.3cm);
\filldraw (1,28) circle (0.3cm);
\filldraw (13,25) circle (0.3cm);
\filldraw (27,9) circle (0.3cm);
\filldraw (17,23) circle (0.3cm);
\filldraw (29,1) circle (0.3cm);
\filldraw (1,29) circle (0.3cm);
\filldraw (3,29) circle (0.3cm);
\filldraw (29,3) circle (0.3cm);
\filldraw (27,11) circle (0.3cm);
\filldraw (15,25) circle (0.3cm);
\filldraw (5,29) circle (0.3cm);
\filldraw (13,27) circle (0.3cm);
\filldraw (30,1) circle (0.3cm);
\filldraw (1,30) circle (0.3cm);
\filldraw (17,25) circle (0.3cm);
\filldraw (25,17) circle (0.3cm);
\filldraw (29,9) circle (0.3cm);
\filldraw (15,27) circle (0.3cm);
\filldraw (31,1) circle (0.3cm);
\filldraw (1,31) circle (0.3cm);
\filldraw (29,11) circle (0.3cm);
\filldraw (3,31) circle (0.3cm);
\filldraw (31,3) circle (0.3cm);
\filldraw (5,31) circle (0.3cm);
\filldraw (25,19) circle (0.3cm);
\filldraw (13,29) circle (0.3cm);
\filldraw (17,27) circle (0.3cm);
\filldraw (27,17) circle (0.3cm);
\filldraw (32,1) circle (0.3cm);
\filldraw (1,32) circle (0.3cm);
\filldraw (31,9) circle (0.3cm);
\filldraw (15,29) circle (0.3cm);
\filldraw (25,21) circle (0.3cm);
\filldraw (31,11) circle (0.3cm);
\filldraw (33,1) circle (0.3cm);
\filldraw (1,33) circle (0.3cm);
\filldraw (27,19) circle (0.3cm);
\filldraw (3,33) circle (0.3cm);
\filldraw (33,3) circle (0.3cm);
\filldraw (5,33) circle (0.3cm);
\filldraw (13,31) circle (0.3cm);
\filldraw (17,29) circle (0.3cm);
\filldraw (29,17) circle (0.3cm);
\filldraw (25,23) circle (0.3cm);
\filldraw (34,1) circle (0.3cm);
\filldraw (1,34) circle (0.3cm);
\filldraw (33,9) circle (0.3cm);
\filldraw (27,21) circle (0.3cm);
\filldraw (15,31) circle (0.3cm);
\filldraw (29,19) circle (0.3cm);
\filldraw (33,11) circle (0.3cm);
\filldraw (35,1) circle (0.3cm);
\filldraw (1,35) circle (0.3cm);
\filldraw (3,35) circle (0.3cm);
\filldraw (35,3) circle (0.3cm);
\filldraw (5,35) circle (0.3cm);
\filldraw (17,31) circle (0.3cm);
\filldraw (31,17) circle (0.3cm);
\filldraw (25,25) circle (0.3cm);
\filldraw (13,33) circle (0.3cm);
\filldraw (27,23) circle (0.3cm);
\filldraw (29,21) circle (0.3cm);
\filldraw (36,1) circle (0.3cm);
\filldraw (1,36) circle (0.3cm);
\filldraw (35,9) circle (0.3cm);
\filldraw (15,33) circle (0.3cm);
\filldraw (31,19) circle (0.3cm);
\filldraw (35,11) circle (0.3cm);
\filldraw (25,27) circle (0.3cm);
\filldraw (27,25) circle (0.3cm);
\filldraw (37,1) circle (0.3cm);
\filldraw (1,37) circle (0.3cm);
\filldraw (29,23) circle (0.3cm);
\filldraw (3,37) circle (0.3cm);
\filldraw (37,3) circle (0.3cm);
\filldraw (17,33) circle (0.3cm);
\filldraw (33,17) circle (0.3cm);
\filldraw (5,37) circle (0.3cm);
\filldraw (13,35) circle (0.3cm);
\filldraw (38,1) circle (0.3cm);
\filldraw (1,38) circle (0.3cm);
\filldraw (37,9) circle (0.3cm);
\filldraw (15,35) circle (0.3cm);
\filldraw (33,19) circle (0.3cm);
\filldraw (27,27) circle (0.3cm);
\filldraw (25,29) circle (0.3cm);
\filldraw (29,25) circle (0.3cm);
\filldraw (37,11) circle (0.3cm);
\filldraw (17,35) circle (0.3cm);
\filldraw (35,17) circle (0.3cm);
\filldraw (39,1) circle (0.3cm);
\filldraw (1,39) circle (0.3cm);
\filldraw (3,39) circle (0.3cm);
\filldraw (39,3) circle (0.3cm);
\filldraw (13,37) circle (0.3cm);
\filldraw (5,39) circle (0.3cm);
\filldraw (27,29) circle (0.3cm);
\filldraw (29,27) circle (0.3cm);
\filldraw (35,19) circle (0.3cm);
\filldraw (25,31) circle (0.3cm);
\filldraw (15,37) circle (0.3cm);
\filldraw (40,1) circle (0.3cm);
\filldraw (39,9) circle (0.3cm);
\filldraw (39,11) circle (0.3cm);
\filldraw (17,37) circle (0.3cm);
\filldraw (37,17) circle (0.3cm);
\filldraw (41,1) circle (0.3cm);
\filldraw (29,29) circle (0.3cm);
\filldraw (41,3) circle (0.3cm);
\filldraw (13,39) circle (0.3cm);
\filldraw (27,31) circle (0.3cm);
\filldraw (25,33) circle (0.3cm);
\filldraw (37,19) circle (0.3cm);
\filldraw (15,39) circle (0.3cm);
\filldraw (41,9) circle (0.3cm);
\filldraw (42,1) circle (0.3cm);
\filldraw (41,11) circle (0.3cm);
\filldraw (29,31) circle (0.3cm);
\filldraw (17,39) circle (0.3cm);
\filldraw (39,17) circle (0.3cm);
\filldraw (27,33) circle (0.3cm);
\filldraw (43,1) circle (0.3cm);
\filldraw (25,35) circle (0.3cm);
\filldraw (43,3) circle (0.3cm);
\filldraw (39,19) circle (0.3cm);
\filldraw (43,9) circle (0.3cm);
\filldraw (29,33) circle (0.3cm);
\filldraw (44,1) circle (0.3cm);
\filldraw (27,35) circle (0.3cm);
\filldraw (43,11) circle (0.3cm);
\filldraw (41,17) circle (0.3cm);
\filldraw (25,37) circle (0.3cm);
\filldraw (37,25) circle (0.3cm);
\filldraw (45,1) circle (0.3cm);
\filldraw (45,3) circle (0.3cm);
\filldraw (41,19) circle (0.3cm);
\filldraw (29,35) circle (0.3cm);
\filldraw (27,37) circle (0.3cm);
\filldraw (37,27) circle (0.3cm);
\filldraw (45,9) circle (0.3cm);
\filldraw (46,1) circle (0.3cm);
\filldraw (43,17) circle (0.3cm);
\filldraw (45,11) circle (0.3cm);
\filldraw (25,39) circle (0.3cm);
\filldraw (39,25) circle (0.3cm);
\filldraw (47,1) circle (0.3cm);
\filldraw (43,19) circle (0.3cm);
\filldraw (29,37) circle (0.3cm);
\filldraw (37,29) circle (0.3cm);
\filldraw (47,3) circle (0.3cm);
\filldraw (27,39) circle (0.3cm);
\filldraw (39,27) circle (0.3cm);
\filldraw (47,9) circle (0.3cm);
\filldraw (48,1) circle (0.3cm);
\filldraw (41,25) circle (0.3cm);
\filldraw (45,17) circle (0.3cm);
\filldraw (47,11) circle (0.3cm);
\filldraw (37,31) circle (0.3cm);
\filldraw (29,39) circle (0.3cm);
\filldraw (39,29) circle (0.3cm);
\filldraw (45,19) circle (0.3cm);
\filldraw (49,1) circle (0.3cm);
\filldraw (49,3) circle (0.3cm);
\filldraw (41,27) circle (0.3cm);
\filldraw (37,33) circle (0.3cm);
\filldraw (43,25) circle (0.3cm);
\filldraw (49,9) circle (0.3cm);
\filldraw (39,31) circle (0.3cm);
\filldraw (47,17) circle (0.3cm);
\filldraw (50,1) circle (0.3cm);
\filldraw (49,11) circle (0.3cm);
\filldraw (41,29) circle (0.3cm);
\filldraw (47,19) circle (0.3cm);
\filldraw (43,27) circle (0.3cm);
\filldraw (37,35) circle (0.3cm);
\filldraw (51,1) circle (0.3cm);
\filldraw (51,3) circle (0.3cm);
\filldraw (39,33) circle (0.3cm);
\filldraw (41,31) circle (0.3cm);
\filldraw (45,25) circle (0.3cm);
\filldraw (51,9) circle (0.3cm);
\filldraw (49,17) circle (0.3cm);
\filldraw (52,1) circle (0.3cm);
\filldraw (51,11) circle (0.3cm);
\filldraw (37,37) circle (0.3cm);
\filldraw (39,35) circle (0.3cm);
\filldraw (45,27) circle (0.3cm);
\filldraw (49,19) circle (0.3cm);
\filldraw (41,33) circle (0.3cm);
\filldraw (53,1) circle (0.3cm);
\filldraw (53,3) circle (0.3cm);
\filldraw (47,25) circle (0.3cm);
\filldraw (53,9) circle (0.3cm);
\filldraw (51,17) circle (0.3cm);
\filldraw (37,39) circle (0.3cm);
\filldraw (39,37) circle (0.3cm);
\filldraw (41,35) circle (0.3cm);
\filldraw (54,1) circle (0.3cm);
\filldraw (53,11) circle (0.3cm);
\filldraw (47,27) circle (0.3cm);
\filldraw (51,19) circle (0.3cm);
\filldraw (55,1) circle (0.3cm);
\filldraw (49,25) circle (0.3cm);
\filldraw (55,3) circle (0.3cm);
\filldraw (39,39) circle (0.3cm);
\filldraw (41,37) circle (0.3cm);
\filldraw (53,17) circle (0.3cm);
\filldraw (55,9) circle (0.3cm);
\filldraw (49,27) circle (0.3cm);
\filldraw (56,1) circle (0.3cm);
\filldraw (55,11) circle (0.3cm);
\filldraw (53,19) circle (0.3cm);
\filldraw (41,39) circle (0.3cm);
\filldraw (51,25) circle (0.3cm);
\filldraw (57,1) circle (0.3cm);
\filldraw (57,3) circle (0.3cm);
\filldraw (55,17) circle (0.3cm);
\filldraw (57,9) circle (0.3cm);
\filldraw (51,27) circle (0.3cm);
\filldraw (57,11) circle (0.3cm);
\filldraw (55,19) circle (0.3cm);
\filldraw (53,25) circle (0.3cm);
\filldraw (49,33) circle (0.3cm);
\filldraw (57,17) circle (0.3cm);
\filldraw (53,27) circle (0.3cm);
\filldraw (57,19) circle (0.3cm);
\filldraw (49,35) circle (0.3cm);
\filldraw (55,25) circle (0.3cm);
\filldraw (51,33) circle (0.3cm);
\filldraw (55,27) circle (0.3cm);
\filldraw (49,37) circle (0.3cm);
\filldraw (51,35) circle (0.3cm);
\filldraw (57,25) circle (0.3cm);
\filldraw (53,33) circle (0.3cm);
\filldraw (49,39) circle (0.3cm);
\filldraw (51,37) circle (0.3cm);
\filldraw (57,27) circle (0.3cm);
\filldraw (53,35) circle (0.3cm);
\filldraw (55,33) circle (0.3cm);
\filldraw (51,39) circle (0.3cm);
\filldraw (53,37) circle (0.3cm);
\filldraw (55,35) circle (0.3cm);
\filldraw (53,39) circle (0.3cm);
\filldraw (57,33) circle (0.3cm);
\filldraw (57,35) circle (0.3cm);
\end{tikzpicture}
\captionsetup{width = 0.95\textwidth}
\caption{The set arising from $\left\{(1,0), (0,1), (6,4)\right\}$ with $L-$shapes. Like in Figure 7, the interior region splits into two regions each containing a lattice.}
\label{fig:Lshape}
\end{figure}
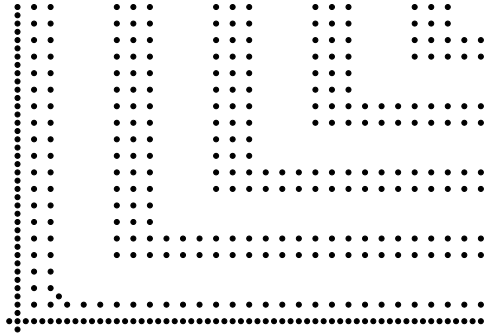
\end{center}
\vspace{-30pt}
\item If $m$ is even and $n>3$ is odd, then the set is the usual lattice coming from $\{(1,0), (0,1)\}$ for all $x \notin \left\{ m, m+1 \right\}$.  Clearly, at $x=m$, we have the extra point $(m,n)$, and the column (otherwise of period 2) at $x=m+1$ is truncated at $y = n$ because of the extra sums generated by the interaction of $(m,n)$ and the column over $x=1$.  Then, for larger values of $x$, the structure returns to the same lattice as before because for $(m,n)$ to interfere, it would have to be summed with a vector with odd $x-$coordinate and even $y-$coordinate, but these vectors only exist at $x=1$, which we have already discussed.
\item If $m$ is even and $n=3$, then the set is the usual lattice created by $\{(1,0), (0,1)\}$ set for all $x<m$.  There is no element $(x,y)$ at $x=m$ for $y>3$.  For $x>m$, the normal lattice is shifted to the right: all subsequent elements have even $x-$coordinate and odd $y-$coordinate.
\end{enumerate}

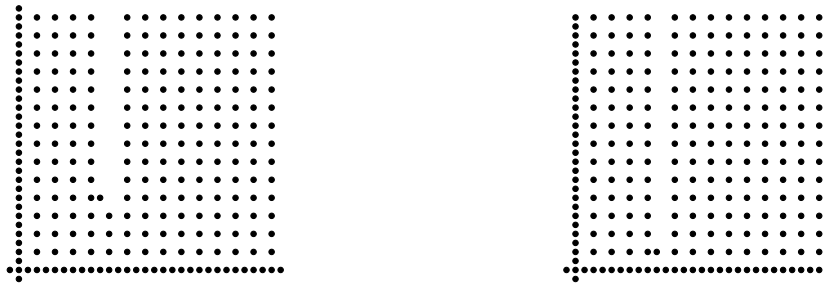
\begin{figure}[h!]
\begin{minipage}[b]{0.48\linewidth}
\centering
\begin{tikzpicture}[scale=0.12]
\filldraw (1,0) circle (0.3cm);
\filldraw (0,1) circle (0.3cm);
\filldraw (10,9) circle (0.3cm);
\filldraw (1,1) circle (0.3cm);
\filldraw (2,1) circle (0.3cm);
\filldraw (1,2) circle (0.3cm);
\filldraw (3,1) circle (0.3cm);
\filldraw (1,3) circle (0.3cm);
\filldraw (4,1) circle (0.3cm);
\filldraw (1,4) circle (0.3cm);
\filldraw (3,3) circle (0.3cm);
\filldraw (5,1) circle (0.3cm);
\filldraw (1,5) circle (0.3cm);
\filldraw (3,5) circle (0.3cm);
\filldraw (5,3) circle (0.3cm);
\filldraw (6,1) circle (0.3cm);
\filldraw (1,6) circle (0.3cm);
\filldraw (7,1) circle (0.3cm);
\filldraw (1,7) circle (0.3cm);
\filldraw (5,5) circle (0.3cm);
\filldraw (3,7) circle (0.3cm);
\filldraw (7,3) circle (0.3cm);
\filldraw (8,1) circle (0.3cm);
\filldraw (1,8) circle (0.3cm);
\filldraw (5,7) circle (0.3cm);
\filldraw (7,5) circle (0.3cm);
\filldraw (9,1) circle (0.3cm);
\filldraw (1,9) circle (0.3cm);
\filldraw (3,9) circle (0.3cm);
\filldraw (9,3) circle (0.3cm);
\filldraw (7,7) circle (0.3cm);
\filldraw (10,1) circle (0.3cm);
\filldraw (1,10) circle (0.3cm);
\filldraw (5,9) circle (0.3cm);
\filldraw (9,5) circle (0.3cm);
\filldraw (11,1) circle (0.3cm);
\filldraw (1,11) circle (0.3cm);
\filldraw (3,11) circle (0.3cm);
\filldraw (11,3) circle (0.3cm);
\filldraw (7,9) circle (0.3cm);
\filldraw (9,7) circle (0.3cm);
\filldraw (12,1) circle (0.3cm);
\filldraw (1,12) circle (0.3cm);
\filldraw (5,11) circle (0.3cm);
\filldraw (11,5) circle (0.3cm);
\filldraw (9,9) circle (0.3cm);
\filldraw (13,1) circle (0.3cm);
\filldraw (1,13) circle (0.3cm);
\filldraw (7,11) circle (0.3cm);
\filldraw (11,7) circle (0.3cm);
\filldraw (3,13) circle (0.3cm);
\filldraw (13,3) circle (0.3cm);
\filldraw (5,13) circle (0.3cm);
\filldraw (13,5) circle (0.3cm);
\filldraw (14,1) circle (0.3cm);
\filldraw (1,14) circle (0.3cm);
\filldraw (9,11) circle (0.3cm);
\filldraw (7,13) circle (0.3cm);
\filldraw (13,7) circle (0.3cm);
\filldraw (15,1) circle (0.3cm);
\filldraw (1,15) circle (0.3cm);
\filldraw (3,15) circle (0.3cm);
\filldraw (15,3) circle (0.3cm);
\filldraw (5,15) circle (0.3cm);
\filldraw (15,5) circle (0.3cm);
\filldraw (9,13) circle (0.3cm);
\filldraw (13,9) circle (0.3cm);
\filldraw (16,1) circle (0.3cm);
\filldraw (1,16) circle (0.3cm);
\filldraw (7,15) circle (0.3cm);
\filldraw (15,7) circle (0.3cm);
\filldraw (17,1) circle (0.3cm);
\filldraw (1,17) circle (0.3cm);
\filldraw (13,11) circle (0.3cm);
\filldraw (3,17) circle (0.3cm);
\filldraw (17,3) circle (0.3cm);
\filldraw (9,15) circle (0.3cm);
\filldraw (15,9) circle (0.3cm);
\filldraw (5,17) circle (0.3cm);
\filldraw (17,5) circle (0.3cm);
\filldraw (18,1) circle (0.3cm);
\filldraw (1,18) circle (0.3cm);
\filldraw (7,17) circle (0.3cm);
\filldraw (17,7) circle (0.3cm);
\filldraw (13,13) circle (0.3cm);
\filldraw (15,11) circle (0.3cm);
\filldraw (19,1) circle (0.3cm);
\filldraw (1,19) circle (0.3cm);
\filldraw (3,19) circle (0.3cm);
\filldraw (19,3) circle (0.3cm);
\filldraw (9,17) circle (0.3cm);
\filldraw (17,9) circle (0.3cm);
\filldraw (5,19) circle (0.3cm);
\filldraw (19,5) circle (0.3cm);
\filldraw (13,15) circle (0.3cm);
\filldraw (15,13) circle (0.3cm);
\filldraw (20,1) circle (0.3cm);
\filldraw (1,20) circle (0.3cm);
\filldraw (7,19) circle (0.3cm);
\filldraw (19,7) circle (0.3cm);
\filldraw (17,11) circle (0.3cm);
\filldraw (21,1) circle (0.3cm);
\filldraw (1,21) circle (0.3cm);
\filldraw (9,19) circle (0.3cm);
\filldraw (19,9) circle (0.3cm);
\filldraw (3,21) circle (0.3cm);
\filldraw (21,3) circle (0.3cm);
\filldraw (15,15) circle (0.3cm);
\filldraw (13,17) circle (0.3cm);
\filldraw (17,13) circle (0.3cm);
\filldraw (5,21) circle (0.3cm);
\filldraw (21,5) circle (0.3cm);
\filldraw (19,11) circle (0.3cm);
\filldraw (22,1) circle (0.3cm);
\filldraw (1,22) circle (0.3cm);
\filldraw (7,21) circle (0.3cm);
\filldraw (21,7) circle (0.3cm);
\filldraw (15,17) circle (0.3cm);
\filldraw (17,15) circle (0.3cm);
\filldraw (9,21) circle (0.3cm);
\filldraw (21,9) circle (0.3cm);
\filldraw (23,1) circle (0.3cm);
\filldraw (1,23) circle (0.3cm);
\filldraw (13,19) circle (0.3cm);
\filldraw (19,13) circle (0.3cm);
\filldraw (3,23) circle (0.3cm);
\filldraw (23,3) circle (0.3cm);
\filldraw (5,23) circle (0.3cm);
\filldraw (23,5) circle (0.3cm);
\filldraw (21,11) circle (0.3cm);
\filldraw (24,1) circle (0.3cm);
\filldraw (1,24) circle (0.3cm);
\filldraw (7,23) circle (0.3cm);
\filldraw (23,7) circle (0.3cm);
\filldraw (17,17) circle (0.3cm);
\filldraw (15,19) circle (0.3cm);
\filldraw (19,15) circle (0.3cm);
\filldraw (9,23) circle (0.3cm);
\filldraw (23,9) circle (0.3cm);
\filldraw (13,21) circle (0.3cm);
\filldraw (21,13) circle (0.3cm);
\filldraw (25,1) circle (0.3cm);
\filldraw (1,25) circle (0.3cm);
\filldraw (3,25) circle (0.3cm);
\filldraw (25,3) circle (0.3cm);
\filldraw (5,25) circle (0.3cm);
\filldraw (25,5) circle (0.3cm);
\filldraw (23,11) circle (0.3cm);
\filldraw (17,19) circle (0.3cm);
\filldraw (19,17) circle (0.3cm);
\filldraw (15,21) circle (0.3cm);
\filldraw (21,15) circle (0.3cm);
\filldraw (7,25) circle (0.3cm);
\filldraw (25,7) circle (0.3cm);
\filldraw (26,1) circle (0.3cm);
\filldraw (1,26) circle (0.3cm);
\filldraw (13,23) circle (0.3cm);
\filldraw (23,13) circle (0.3cm);
\filldraw (9,25) circle (0.3cm);
\filldraw (25,9) circle (0.3cm);
\filldraw (19,19) circle (0.3cm);
\filldraw (27,1) circle (0.3cm);
\filldraw (1,27) circle (0.3cm);
\filldraw (17,21) circle (0.3cm);
\filldraw (21,17) circle (0.3cm);
\filldraw (3,27) circle (0.3cm);
\filldraw (27,3) circle (0.3cm);
\filldraw (25,11) circle (0.3cm);
\filldraw (5,27) circle (0.3cm);
\filldraw (27,5) circle (0.3cm);
\filldraw (15,23) circle (0.3cm);
\filldraw (23,15) circle (0.3cm);
\filldraw (7,27) circle (0.3cm);
\filldraw (27,7) circle (0.3cm);
\filldraw (28,1) circle (0.3cm);
\filldraw (1,28) circle (0.3cm);
\filldraw (13,25) circle (0.3cm);
\filldraw (25,13) circle (0.3cm);
\filldraw (19,21) circle (0.3cm);
\filldraw (21,19) circle (0.3cm);
\filldraw (9,27) circle (0.3cm);
\filldraw (27,9) circle (0.3cm);
\filldraw (17,23) circle (0.3cm);
\filldraw (23,17) circle (0.3cm);
\filldraw (29,1) circle (0.3cm);
\filldraw (1,29) circle (0.3cm);
\filldraw (3,29) circle (0.3cm);
\filldraw (29,3) circle (0.3cm);
\filldraw (27,11) circle (0.3cm);
\filldraw (15,25) circle (0.3cm);
\filldraw (25,15) circle (0.3cm);
\filldraw (5,29) circle (0.3cm);
\filldraw (29,5) circle (0.3cm);
\filldraw (21,21) circle (0.3cm);
\filldraw (7,29) circle (0.3cm);
\filldraw (29,7) circle (0.3cm);
\filldraw (19,23) circle (0.3cm);
\filldraw (23,19) circle (0.3cm);
\filldraw (13,27) circle (0.3cm);
\filldraw (27,13) circle (0.3cm);
\filldraw (30,1) circle (0.3cm);
\filldraw (1,30) circle (0.3cm);
\filldraw (17,25) circle (0.3cm);
\filldraw (25,17) circle (0.3cm);
\filldraw (9,29) circle (0.3cm);
\filldraw (29,9) circle (0.3cm);
\filldraw (15,27) circle (0.3cm);
\filldraw (27,15) circle (0.3cm);
\filldraw (29,11) circle (0.3cm);
\filldraw (21,23) circle (0.3cm);
\filldraw (23,21) circle (0.3cm);
\filldraw (19,25) circle (0.3cm);
\filldraw (25,19) circle (0.3cm);
\filldraw (13,29) circle (0.3cm);
\filldraw (29,13) circle (0.3cm);
\filldraw (17,27) circle (0.3cm);
\filldraw (27,17) circle (0.3cm);
\filldraw (23,23) circle (0.3cm);
\filldraw (15,29) circle (0.3cm);
\filldraw (29,15) circle (0.3cm);
\filldraw (21,25) circle (0.3cm);
\filldraw (25,21) circle (0.3cm);
\filldraw (19,27) circle (0.3cm);
\filldraw (27,19) circle (0.3cm);
\filldraw (17,29) circle (0.3cm);
\filldraw (29,17) circle (0.3cm);
\filldraw (23,25) circle (0.3cm);
\filldraw (25,23) circle (0.3cm);
\filldraw (21,27) circle (0.3cm);
\filldraw (27,21) circle (0.3cm);
\filldraw (19,29) circle (0.3cm);
\filldraw (29,19) circle (0.3cm);
\filldraw (25,25) circle (0.3cm);
\filldraw (23,27) circle (0.3cm);
\filldraw (27,23) circle (0.3cm);
\filldraw (21,29) circle (0.3cm);
\filldraw (29,21) circle (0.3cm);
\filldraw (25,27) circle (0.3cm);
\filldraw (27,25) circle (0.3cm);
\filldraw (23,29) circle (0.3cm);
\filldraw (29,23) circle (0.3cm);
\filldraw (27,27) circle (0.3cm);
\filldraw (25,29) circle (0.3cm);
\filldraw (29,25) circle (0.3cm);
\filldraw (27,29) circle (0.3cm);
\filldraw (29,27) circle (0.3cm);
\filldraw (29,29) circle (0.3cm);
\end{tikzpicture}
\end{minipage}
\begin{minipage}[b]{0.48\linewidth}
\centering
\begin{tikzpicture}[scale=0.12]

\filldraw (1,0) circle (0.3cm);
\filldraw (0,1) circle (0.3cm);
\filldraw (10,3) circle (0.3cm);
\filldraw (1,1) circle (0.3cm);
\filldraw (2,1) circle (0.3cm);
\filldraw (1,2) circle (0.3cm);
\filldraw (3,1) circle (0.3cm);
\filldraw (1,3) circle (0.3cm);
\filldraw (4,1) circle (0.3cm);
\filldraw (1,4) circle (0.3cm);
\filldraw (3,3) circle (0.3cm);
\filldraw (5,1) circle (0.3cm);
\filldraw (1,5) circle (0.3cm);
\filldraw (3,5) circle (0.3cm);
\filldraw (5,3) circle (0.3cm);
\filldraw (6,1) circle (0.3cm);
\filldraw (1,6) circle (0.3cm);
\filldraw (7,1) circle (0.3cm);
\filldraw (1,7) circle (0.3cm);
\filldraw (5,5) circle (0.3cm);
\filldraw (3,7) circle (0.3cm);
\filldraw (7,3) circle (0.3cm);
\filldraw (8,1) circle (0.3cm);
\filldraw (1,8) circle (0.3cm);
\filldraw (5,7) circle (0.3cm);
\filldraw (7,5) circle (0.3cm);
\filldraw (9,1) circle (0.3cm);
\filldraw (1,9) circle (0.3cm);
\filldraw (3,9) circle (0.3cm);
\filldraw (9,3) circle (0.3cm);
\filldraw (7,7) circle (0.3cm);
\filldraw (10,1) circle (0.3cm);
\filldraw (1,10) circle (0.3cm);
\filldraw (5,9) circle (0.3cm);
\filldraw (9,5) circle (0.3cm);
\filldraw (11,1) circle (0.3cm);
\filldraw (1,11) circle (0.3cm);
\filldraw (3,11) circle (0.3cm);
\filldraw (7,9) circle (0.3cm);
\filldraw (9,7) circle (0.3cm);
\filldraw (12,1) circle (0.3cm);
\filldraw (1,12) circle (0.3cm);
\filldraw (5,11) circle (0.3cm);
\filldraw (12,3) circle (0.3cm);
\filldraw (9,9) circle (0.3cm);
\filldraw (12,5) circle (0.3cm);
\filldraw (13,1) circle (0.3cm);
\filldraw (1,13) circle (0.3cm);
\filldraw (7,11) circle (0.3cm);
\filldraw (3,13) circle (0.3cm);
\filldraw (12,7) circle (0.3cm);
\filldraw (5,13) circle (0.3cm);
\filldraw (14,1) circle (0.3cm);
\filldraw (1,14) circle (0.3cm);
\filldraw (9,11) circle (0.3cm);
\filldraw (14,3) circle (0.3cm);
\filldraw (7,13) circle (0.3cm);
\filldraw (14,5) circle (0.3cm);
\filldraw (12,9) circle (0.3cm);
\filldraw (15,1) circle (0.3cm);
\filldraw (1,15) circle (0.3cm);
\filldraw (3,15) circle (0.3cm);
\filldraw (14,7) circle (0.3cm);
\filldraw (5,15) circle (0.3cm);
\filldraw (9,13) circle (0.3cm);
\filldraw (16,1) circle (0.3cm);
\filldraw (1,16) circle (0.3cm);
\filldraw (16,3) circle (0.3cm);
\filldraw (12,11) circle (0.3cm);
\filldraw (7,15) circle (0.3cm);
\filldraw (14,9) circle (0.3cm);
\filldraw (16,5) circle (0.3cm);
\filldraw (17,1) circle (0.3cm);
\filldraw (1,17) circle (0.3cm);
\filldraw (3,17) circle (0.3cm);
\filldraw (16,7) circle (0.3cm);
\filldraw (9,15) circle (0.3cm);
\filldraw (12,13) circle (0.3cm);
\filldraw (5,17) circle (0.3cm);
\filldraw (14,11) circle (0.3cm);
\filldraw (18,1) circle (0.3cm);
\filldraw (1,18) circle (0.3cm);
\filldraw (18,3) circle (0.3cm);
\filldraw (16,9) circle (0.3cm);
\filldraw (7,17) circle (0.3cm);
\filldraw (18,5) circle (0.3cm);
\filldraw (19,1) circle (0.3cm);
\filldraw (1,19) circle (0.3cm);
\filldraw (14,13) circle (0.3cm);
\filldraw (12,15) circle (0.3cm);
\filldraw (3,19) circle (0.3cm);
\filldraw (9,17) circle (0.3cm);
\filldraw (18,7) circle (0.3cm);
\filldraw (16,11) circle (0.3cm);
\filldraw (5,19) circle (0.3cm);
\filldraw (20,1) circle (0.3cm);
\filldraw (1,20) circle (0.3cm);
\filldraw (18,9) circle (0.3cm);
\filldraw (20,3) circle (0.3cm);
\filldraw (7,19) circle (0.3cm);
\filldraw (14,15) circle (0.3cm);
\filldraw (20,5) circle (0.3cm);
\filldraw (16,13) circle (0.3cm);
\filldraw (12,17) circle (0.3cm);
\filldraw (21,1) circle (0.3cm);
\filldraw (1,21) circle (0.3cm);
\filldraw (9,19) circle (0.3cm);
\filldraw (18,11) circle (0.3cm);
\filldraw (20,7) circle (0.3cm);
\filldraw (3,21) circle (0.3cm);
\filldraw (5,21) circle (0.3cm);
\filldraw (20,9) circle (0.3cm);
\filldraw (16,15) circle (0.3cm);
\filldraw (22,1) circle (0.3cm);
\filldraw (1,22) circle (0.3cm);
\filldraw (14,17) circle (0.3cm);
\filldraw (7,21) circle (0.3cm);
\filldraw (22,3) circle (0.3cm);
\filldraw (18,13) circle (0.3cm);
\filldraw (12,19) circle (0.3cm);
\filldraw (22,5) circle (0.3cm);
\filldraw (20,11) circle (0.3cm);
\filldraw (9,21) circle (0.3cm);
\filldraw (23,1) circle (0.3cm);
\filldraw (1,23) circle (0.3cm);
\filldraw (22,7) circle (0.3cm);
\filldraw (3,23) circle (0.3cm);
\filldraw (16,17) circle (0.3cm);
\filldraw (18,15) circle (0.3cm);
\filldraw (5,23) circle (0.3cm);
\filldraw (14,19) circle (0.3cm);
\filldraw (22,9) circle (0.3cm);
\filldraw (20,13) circle (0.3cm);
\filldraw (24,1) circle (0.3cm);
\filldraw (1,24) circle (0.3cm);
\filldraw (7,23) circle (0.3cm);
\filldraw (24,3) circle (0.3cm);
\filldraw (12,21) circle (0.3cm);
\filldraw (24,5) circle (0.3cm);
\filldraw (22,11) circle (0.3cm);
\filldraw (9,23) circle (0.3cm);
\filldraw (18,17) circle (0.3cm);
\filldraw (16,19) circle (0.3cm);
\filldraw (24,7) circle (0.3cm);
\filldraw (20,15) circle (0.3cm);
\filldraw (25,1) circle (0.3cm);
\filldraw (1,25) circle (0.3cm);
\filldraw (3,25) circle (0.3cm);
\filldraw (14,21) circle (0.3cm);
\filldraw (5,25) circle (0.3cm);
\filldraw (22,13) circle (0.3cm);
\filldraw (24,9) circle (0.3cm);
\filldraw (12,23) circle (0.3cm);
\filldraw (7,25) circle (0.3cm);
\filldraw (26,1) circle (0.3cm);
\filldraw (1,26) circle (0.3cm);
\filldraw (26,3) circle (0.3cm);
\filldraw (18,19) circle (0.3cm);
\filldraw (20,17) circle (0.3cm);
\filldraw (24,11) circle (0.3cm);
\filldraw (16,21) circle (0.3cm);
\filldraw (26,5) circle (0.3cm);
\filldraw (9,25) circle (0.3cm);
\filldraw (22,15) circle (0.3cm);
\filldraw (26,7) circle (0.3cm);
\filldraw (14,23) circle (0.3cm);
\filldraw (27,1) circle (0.3cm);
\filldraw (1,27) circle (0.3cm);
\filldraw (3,27) circle (0.3cm);
\filldraw (24,13) circle (0.3cm);
\filldraw (5,27) circle (0.3cm);
\filldraw (26,9) circle (0.3cm);
\filldraw (20,19) circle (0.3cm);
\filldraw (18,21) circle (0.3cm);
\filldraw (12,25) circle (0.3cm);
\filldraw (22,17) circle (0.3cm);
\filldraw (7,27) circle (0.3cm);
\filldraw (28,1) circle (0.3cm);
\filldraw (1,28) circle (0.3cm);
\filldraw (16,23) circle (0.3cm);
\filldraw (28,3) circle (0.3cm);
\filldraw (26,11) circle (0.3cm);
\filldraw (24,15) circle (0.3cm);
\filldraw (28,5) circle (0.3cm);
\filldraw (9,27) circle (0.3cm);
\filldraw (14,25) circle (0.3cm);
\filldraw (28,7) circle (0.3cm);
\filldraw (20,21) circle (0.3cm);
\filldraw (1,29) circle (0.3cm);
\filldraw (26,13) circle (0.3cm);
\filldraw (22,19) circle (0.3cm);
\filldraw (3,29) circle (0.3cm);
\filldraw (18,23) circle (0.3cm);
\filldraw (28,9) circle (0.3cm);
\filldraw (24,17) circle (0.3cm);
\filldraw (5,29) circle (0.3cm);
\filldraw (12,27) circle (0.3cm);
\filldraw (16,25) circle (0.3cm);
\filldraw (7,29) circle (0.3cm);
\filldraw (26,15) circle (0.3cm);
\filldraw (28,11) circle (0.3cm);
\filldraw (9,29) circle (0.3cm);
\filldraw (14,27) circle (0.3cm);
\filldraw (22,21) circle (0.3cm);
\filldraw (20,23) circle (0.3cm);
\filldraw (24,19) circle (0.3cm);
\filldraw (18,25) circle (0.3cm);
\filldraw (28,13) circle (0.3cm);
\filldraw (26,17) circle (0.3cm);
\filldraw (12,29) circle (0.3cm);
\filldraw (16,27) circle (0.3cm);
\filldraw (28,15) circle (0.3cm);
\filldraw (22,23) circle (0.3cm);
\filldraw (24,21) circle (0.3cm);
\filldraw (20,25) circle (0.3cm);
\filldraw (14,29) circle (0.3cm);
\filldraw (26,19) circle (0.3cm);
\filldraw (18,27) circle (0.3cm);
\filldraw (28,17) circle (0.3cm);
\filldraw (16,29) circle (0.3cm);
\filldraw (24,23) circle (0.3cm);
\filldraw (22,25) circle (0.3cm);
\filldraw (26,21) circle (0.3cm);
\filldraw (20,27) circle (0.3cm);
\filldraw (28,19) circle (0.3cm);
\filldraw (18,29) circle (0.3cm);
\filldraw (24,25) circle (0.3cm);
\filldraw (26,23) circle (0.3cm);
\filldraw (22,27) circle (0.3cm);
\filldraw (28,21) circle (0.3cm);
\filldraw (20,29) circle (0.3cm);
\filldraw (26,25) circle (0.3cm);
\filldraw (24,27) circle (0.3cm);
\filldraw (28,23) circle (0.3cm);
\filldraw (22,29) circle (0.3cm);
\filldraw (26,27) circle (0.3cm);
\filldraw (28,25) circle (0.3cm);
\filldraw (24,29) circle (0.3cm);
\filldraw (28,27) circle (0.3cm);
\filldraw (26,29) circle (0.3cm);
\filldraw (28,29) circle (0.3cm);
\end{tikzpicture}
\end{minipage}
\captionsetup{width = 0.95\textwidth}
\caption{The graphs arising from $\left\{(1,0), (0,1), (10,9)\right\}$, with the lattice temporarily disrupted but ultimately unchanged lattice (left), and $\left\{(1,0), (0,1), (10,3)\right\}$, with the lattice shifted (right).}
\label{fig:block}
\end{figure}
\vspace{-13pt}

\subsection{More Initial Conditions.} Our active investigation was mainly restricted to initial sets containing three vectors in $\mathbb{Z}_{\geq 0}^2$. The investigation of larger sets of initial
vectors seems like a daunting but also very promising avenue for further research.

\subsection{Ulam's Variant.} We recall that Ulam originally proposed another variant on the hexagonal lattice based on the idea of `including a new point if it forms with two previously
defined points the third vertex of a triangle, but only doing it in the case where it is uniquely so related to a previous pair'. We have not investigated this more geometrical variant.

\subsection{Other Variants.} It is clear that Ulam sets can be considered on other algebraic structures equipped with a binary operation and some notion of `size' (though,
as observed above in the case of $\mathbb{R}^{n}_{\geq 0}$, the dependence on the notion of size is rather weak, and the sets are fairly universal).\\

\textbf{Multiplication Over the Complex Numbers.}
One natural example that
comes to mind is $\mathbb{C}$ equipped with multiplication for elements with norm larger than $1$. The introduction of polar coordinates shows that
$$ \left( r_1 \measuredangle \phi_1\right) \left( r_2 \measuredangle \phi_2\right) =   r_1r_2 \measuredangle (\phi_1 + \phi_2),$$
which allows us to reduce to transform the problem to $\mathbb{R}_{>1} \times \mathbb{T}$. Moreover, as above, we can replace multiplication by addition via the logarithm. For an initial set $\left\{v_1, \dots, v_k \right\} \subset \mathbb{R}_{>0} \times \mathbb{T}$, we can use as our notion of length $f(x,y)=x$.\\

\textbf{Ulam Sets in $\mathbb{Z}_{\geq 0} \times \mathbb{Z}_n$.} One particularly natural setting, inspired by multiplication in $\mathbb{C}$, is that of $\mathbb{Z}_{\geq 0} \times \mathbb{Z}_n$. 
\begin{center}
\begin{figure}[h!]
\begin{tikzpicture}[scale=0.15]
\filldraw (1,3) circle (0.25cm);
\filldraw (3,4) circle (0.25cm);
\filldraw (4,1) circle (0.25cm);
\filldraw (5,4) circle (0.25cm);
\filldraw (6,1) circle (0.25cm);
\filldraw (7,4) circle (0.25cm);
\filldraw (7,5) circle (0.25cm);
\filldraw (8,1) circle (0.25cm);
\filldraw (9,4) circle (0.25cm);
\filldraw (10,1) circle (0.25cm);
\filldraw (10,3) circle (0.25cm);
\filldraw (11,4) circle (0.25cm);
\filldraw (12,1) circle (0.25cm);
\filldraw (12,3) circle (0.25cm);
\filldraw (13,4) circle (0.25cm);
\filldraw (13,1) circle (0.25cm);
\filldraw (14,1) circle (0.25cm);
\filldraw (14,3) circle (0.25cm);
\filldraw (15,4) circle (0.25cm);
\filldraw (16,1) circle (0.25cm);
\filldraw (16,5) circle (0.25cm);
\filldraw (16,3) circle (0.25cm);
\filldraw (17,4) circle (0.25cm);
\filldraw (18,1) circle (0.25cm);
\filldraw (18,5) circle (0.25cm);
\filldraw (18,3) circle (0.25cm);
\filldraw (19,4) circle (0.25cm);
\filldraw (19,3) circle (0.25cm);
\filldraw (20,1) circle (0.25cm);
\filldraw (20,5) circle (0.25cm);
\filldraw (20,3) circle (0.25cm);
\filldraw (21,4) circle (0.25cm);
\filldraw (22,3) circle (0.25cm);
\filldraw (22,5) circle (0.25cm);
\filldraw (24,1) circle (0.25cm);
\filldraw (24,3) circle (0.25cm);
\filldraw (24,5) circle (0.25cm);
\filldraw (26,1) circle (0.25cm);
\filldraw (26,3) circle (0.25cm);
\filldraw (26,5) circle (0.25cm);
\filldraw (28,3) circle (0.25cm);
\filldraw (28,1) circle (0.25cm);
\filldraw (28,5) circle (0.25cm);
\filldraw (30,1) circle (0.25cm);
\filldraw (30,5) circle (0.25cm);
\filldraw (32,1) circle (0.25cm);
\filldraw (32,5) circle (0.25cm);
\filldraw (34,5) circle (0.25cm);
\filldraw (34,1) circle (0.25cm);
\filldraw (36,1) circle (0.25cm);
\filldraw (38,1) circle (0.25cm);
\filldraw (40,1) circle (0.25cm);
\filldraw (43,4) circle (0.25cm);
\filldraw (44,1) circle (0.25cm);
\filldraw (47,4) circle (0.25cm);
\filldraw (48,1) circle (0.25cm);
\filldraw (49,1) circle (0.25cm);
\filldraw (50,3) circle (0.25cm);
\filldraw (51,4) circle (0.25cm);
\filldraw (52,1) circle (0.25cm);
\filldraw (52,5) circle (0.25cm);
\filldraw (54,5) circle (0.25cm);
\filldraw (54,3) circle (0.25cm);
\filldraw (55,4) circle (0.25cm);
\filldraw (56,1) circle (0.25cm);
\filldraw (58,3) circle (0.25cm);
\filldraw (58,5) circle (0.25cm);
\filldraw (62,3) circle (0.25cm);
\filldraw (62,5) circle (0.25cm);
\filldraw (62,1) circle (0.25cm);
\filldraw (66,5) circle (0.25cm);
\filldraw (66,1) circle (0.25cm);
\filldraw (70,1) circle (0.25cm);
\filldraw (70,5) circle (0.25cm);
\filldraw (74,1) circle (0.25cm);
\filldraw (81,4) circle (0.25cm);
\filldraw (82,1) circle (0.25cm);
\filldraw (83,4) circle (0.25cm);
\filldraw (84,1) circle (0.25cm);
\filldraw (88,3) circle (0.25cm);
\filldraw (89,4) circle (0.25cm);
\filldraw (90,1) circle (0.25cm);
\filldraw (90,3) circle (0.25cm);
\filldraw (91,4) circle (0.25cm);
\filldraw (91,3) circle (0.25cm);
\filldraw (92,1) circle (0.25cm);
\end{tikzpicture}
\captionsetup{width = 0.95\textwidth}
\caption{The set arising from $\left\{(1,3), (3,4)\right\} \subset \mathbb{Z}_{\geq 0} \times \mathbb{Z}_6$. We observe empirically that the $y-$coordinates 0 and 2 seem not to appear. }
\label{fig:twodlattice}
\end{figure}
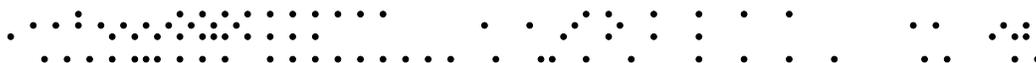
\end{center}
\vspace{-10pt}
 For the Ulam set to be well-defined, it is important that all initial elements $(x,y)$ have $x>0$.  Note also that not all Ulam sets defined this way contain an infinite number of elements (for example, initial vectors given by $\{(1,0), (1,1), (1,2)\} \subset \mathbb{Z}_{\geq 0} \times \mathbb{Z}_3$), and deriving a condition to determine which initial sets exhibit this property is an interesting question for future investigation.
\begin{center}
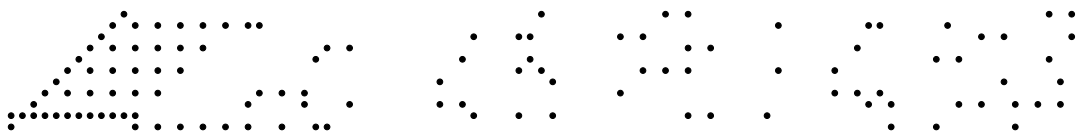
\begin{figure}[h!]
\begin{tikzpicture}[scale=0.15]
\filldraw (1,0) circle (0.25cm);
\filldraw (1,1) circle (0.25cm);
\filldraw (2,1) circle (0.25cm);
\filldraw (3,1) circle (0.25cm);
\filldraw (3,2) circle (0.25cm);
\filldraw (4,1) circle (0.25cm);
\filldraw (4,3) circle (0.25cm);
\filldraw (5,1) circle (0.25cm);
\filldraw (5,4) circle (0.25cm);
\filldraw (6,1) circle (0.25cm);
\filldraw (6,5) circle (0.25cm);
\filldraw (6,3) circle (0.25cm);
\filldraw (7,1) circle (0.25cm);
\filldraw (7,6) circle (0.25cm);
\filldraw (8,1) circle (0.25cm);
\filldraw (8,7) circle (0.25cm);
\filldraw (8,5) circle (0.25cm);
\filldraw (8,3) circle (0.25cm);
\filldraw (9,1) circle (0.25cm);
\filldraw (9,8) circle (0.25cm);
\filldraw (10,1) circle (0.25cm);
\filldraw (10,9) circle (0.25cm);
\filldraw (10,7) circle (0.25cm);
\filldraw (10,3) circle (0.25cm);
\filldraw (10,5) circle (0.25cm);
\filldraw (11,1) circle (0.25cm);
\filldraw (11,10) circle (0.25cm);
\filldraw (12,1) circle (0.25cm);
\filldraw (12,0) circle (0.25cm);
\filldraw (12,9) circle (0.25cm);
\filldraw (12,3) circle (0.25cm);
\filldraw (12,7) circle (0.25cm);
\filldraw (12,5) circle (0.25cm);
\filldraw (14,0) circle (0.25cm);
\filldraw (14,3) circle (0.25cm);
\filldraw (14,9) circle (0.25cm);
\filldraw (14,5) circle (0.25cm);
\filldraw (14,7) circle (0.25cm);
\filldraw (16,0) circle (0.25cm);
\filldraw (16,5) circle (0.25cm);
\filldraw (16,9) circle (0.25cm);
\filldraw (16,7) circle (0.25cm);
\filldraw (18,0) circle (0.25cm);
\filldraw (18,7) circle (0.25cm);
\filldraw (18,9) circle (0.25cm);
\filldraw (20,0) circle (0.25cm);
\filldraw (20,9) circle (0.25cm);
\filldraw (22,2) circle (0.25cm);
\filldraw (22,9) circle (0.25cm);
\filldraw (22,0) circle (0.25cm);
\filldraw (23,9) circle (0.25cm);
\filldraw (23,3) circle (0.25cm);
\filldraw (25,3) circle (0.25cm);
\filldraw (25,0) circle (0.25cm);
\filldraw (27,3) circle (0.25cm);
\filldraw (27,2) circle (0.25cm);
\filldraw (28,6) circle (0.25cm);
\filldraw (28,0) circle (0.25cm);
\filldraw (29,0) circle (0.25cm);
\filldraw (29,7) circle (0.25cm);
\filldraw (31,7) circle (0.25cm);
\filldraw (31,2) circle (0.25cm);
\filldraw (39,4) circle (0.25cm);
\filldraw (39,2) circle (0.25cm);
\filldraw (41,2) circle (0.25cm);
\filldraw (41,6) circle (0.25cm);
\filldraw (42,8) circle (0.25cm);
\filldraw (42,1) circle (0.25cm);
\filldraw (46,5) circle (0.25cm);
\filldraw (46,8) circle (0.25cm);
\filldraw (46,1) circle (0.25cm);
\filldraw (47,8) circle (0.25cm);
\filldraw (47,6) circle (0.25cm);
\filldraw (48,5) circle (0.25cm);
\filldraw (48,10) circle (0.25cm);
\filldraw (49,4) circle (0.25cm);
\filldraw (49,1) circle (0.25cm);
\filldraw (55,3) circle (0.25cm);
\filldraw (55,8) circle (0.25cm);
\filldraw (57,5) circle (0.25cm);
\filldraw (57,8) circle (0.25cm);
\filldraw (59,5) circle (0.25cm);
\filldraw (59,10) circle (0.25cm);
\filldraw (61,5) circle (0.25cm);
\filldraw (61,1) circle (0.25cm);
\filldraw (61,10) circle (0.25cm);
\filldraw (61,7) circle (0.25cm);
\filldraw (63,1) circle (0.25cm);
\filldraw (63,7) circle (0.25cm);
\filldraw (68,1) circle (0.25cm);
\filldraw (69,9) circle (0.25cm);
\filldraw (69,5) circle (0.25cm);
\filldraw (74,3) circle (0.25cm);
\filldraw (74,5) circle (0.25cm);
\filldraw (76,7) circle (0.25cm);
\filldraw (76,3) circle (0.25cm);
\filldraw (77,2) circle (0.25cm);
\filldraw (77,9) circle (0.25cm);
\filldraw (78,9) circle (0.25cm);
\filldraw (78,3) circle (0.25cm);
\filldraw (79,2) circle (0.25cm);
\filldraw (79,0) circle (0.25cm);
\filldraw (83,6) circle (0.25cm);
\filldraw (83,0) circle (0.25cm);
\filldraw (84,9) circle (0.25cm);
\filldraw (85,6) circle (0.25cm);
\filldraw (85,2) circle (0.25cm);
\filldraw (87,8) circle (0.25cm);
\filldraw (87,2) circle (0.25cm);
\filldraw (89,4) circle (0.25cm);
\filldraw (89,8) circle (0.25cm);
\filldraw (90,2) circle (0.25cm);
\filldraw (90,0) circle (0.25cm);
\filldraw (92,2) circle (0.25cm);
\filldraw (93,10) circle (0.25cm);
\filldraw (93,6) circle (0.25cm);
\filldraw (94,4) circle (0.25cm);
\filldraw (94,2) circle (0.25cm);
\filldraw (95,10) circle (0.25cm);
\filldraw (95,8) circle (0.25cm);
\filldraw (96,4) circle (0.25cm);
\end{tikzpicture}
\captionsetup{width = 0.95\textwidth}
\caption{The set arising from $\left\{(1,0), (1,1)\right\} \subset \mathbb{Z}_{\geq 0} \times \mathbb{Z}_{11}$.}
\label{fig:twodlattice}
\end{figure}
\end{center}
\vspace{-20pt}

\textbf{Acknowledgments.}
The authors wish to thank Noah Boorstin for assisting in the creation of computer models and Milo Brandt for pointing out the finite-preimage requirement for $f$-functions.


\begin{thebibliography}{1}
\bibitem{cas} J. Cassaigne and S. Finch, 
A class of 1-additive sequences and quadratic recurrences.
Experiment. Math. 4 (1995), no. 1, 49-60. 

\bibitem{finch1} S. Finch, Conjectures about s-additive sequences. Fibonacci Quart. 29 (1991), no. 3, 209-214. 


\bibitem{finch2} S. Finch, On the regularity of certain 1-additive sequences. J. Combin. Theory Ser. A 60 (1992), no. 1, 123-130. 
 

\bibitem{finch3} S. Finch, Patterns in 1-additive sequences.
Experiment. Math. 1 (1992), no. 1, 57-63.


\bibitem{gibbs} P. Gibbs., An Efficient Method for Computing Ulam Numbers, \textsc{http://vixra.org/abs/1508.0085}.

\bibitem{guy} R. K. Guy, Unsolved problems in number theory, third edition, Problem Books in Mathematics, Springer-Verlag, New York, 2004.

\bibitem{hinman} J. Hinman, B. Kuca, A. Schlesinger and A. Sheydvasser, The Unreasonable Rigidity of Ulam Sets, J. Number Theory, to appear.

\bibitem{knuth} D. Knuth, Documentation of program Ulam-Gibbs, available at\\ www-cs-faculty.stanford.edu/~uno/programs.html.

\bibitem{kuca} B. Kuca, Structures in Additive Sequences, Acta Arith., to appear.

\bibitem{oeis} The On-Line Encyclopedia of Integer Sequences, \textsc{oeis.org}, 27. Nov 2015, Sequence A002858.

\bibitem{q} R. Queneau, Sur les suites s-additives. J. Combin. Theory Ser. A 12 (1972), 31-71. 

\bibitem{recaman}B.  Recaman, Research Problems: Questions on a Sequence of Ulam. 
Amer. Math. Monthly 80 (1973), 919-920. 

\bibitem{ross} D. Ross, The Ulam Sequence and Related Phenomena, PhD Thesis at the University of Wisconsin, Madison, available at github (daniel3735928559).

\bibitem{schm} J. Schmerl and E. Spiegel, The regularity of some 1-additive sequences. 
J. Combin. Theory Ser. A 66 (1994), no. 1, 172-175. 

\bibitem{stein} S. Steinerberger, A Hidden Signal in the Ulam Sequence, Experimental Mathematics, Experiment. Math. 23 (2017), no. 4, 460-467.

\bibitem{strott1} D. Strottman, Some Properties of Ulam Numbers, Los Alamos Technical Report, private communication.

\bibitem{ulam} S. Ulam, Combinatorial analysis in infinite sets and some physical theories.  SIAM Rev. 6 1964 343-355. 

\bibitem{ulam2} S. Ulam, Problems in Modern Mathematics, Science Editions John Wiley \& Sons, Inc., New York 1964.


\end{thebibliography}
\end{document}